\documentclass[reqno,final]{amsproc}

\usepackage{amsmath,amsthm}
\usepackage[all]{xy}
\usepackage[outer=2.2cm,inner=2.2cm]{geometry}
\usepackage{tensor} % for left subscripts that come up in pushouts
\usepackage{manfnt} % provides \mancube
\usepackage{relsize} % provides \mathlarger
\usepackage{microtype}

\usepackage{color}
\definecolor{myurlcolor}{rgb}{0,0,0.4}
\definecolor{mycitecolor}{rgb}{0,0.5,0}
\definecolor{myrefcolor}{rgb}{0.5,0,0}
\usepackage[pagebackref]{hyperref}
\hypersetup{colorlinks,
linkcolor=myrefcolor,
citecolor=mycitecolor,
urlcolor=myurlcolor}

\newtheorem{proposition}{Proposition}[section]
\newtheorem{theorem}[proposition]{Theorem}
\newtheorem{definition}[proposition]{Definition}
\newtheorem{lemma}[proposition]{Lemma}
\newtheorem{corollary}[proposition]{Corollary}
\newtheorem{conjecture}[proposition]{Conjecture}
\newtheorem{problem}[proposition]{Problem}
\theoremstyle{definition}
\newtheorem{example}[proposition]{Example}
\newtheorem{remark}[proposition]{Remark}

\def\be{\begin{equation}}
\def\ee{\end{equation}}
\def\ba{\begin{eqnarray}}
\def\ea{\end{eqnarray}}

\let\originalleft\left
\let\originalright\right
\def\left#1{\mathopen{}\originalleft#1}
\def\right#1{\originalright#1\mathclose{}}

\newcommand{\Nl}{\mathbb{N}}
\newcommand{\Zl}{\mathbb{Z}}

\newcommand{\Rl}{\mathbb{R}}

\newcommand{\Cl}{\mathbb{C}}
\newcommand{\Tl}{\mathbb{T}} % unit circle
\newcommand{\Fl}{\mathbb{F}} % free group
\newcommand{\defin}{:=}
\newcommand{\im}{\mathrm{im}} % image of a map
\newcommand{\implproof}[2]{\underline{\ref{#1}$\Rightarrow$\ref{#2}:}} % useful for proofs of several implications
\newcommand{\id}{\mathrm{id}} % identity morphism
\newcommand{\spc}{\mathrm{sp}} % spectrum of an operator
\newcommand{\Sh}{\mathrm{Sh}} % sheaves
\newcommand{\pushout}[2]{\tensor[_{#1}]{\amalg}{_{#2}}} % pushout symbol labelled by one morphism on each side
\newcommand{\pullback}[2]{\tensor[_{#1}]{\times}{_{#2}}} % pullback symbol labelled by one morphism on each side
\newcommand{\tr}{\mathrm{tr}} % trace
\newcommand{\op}{\mathrm{op}} % opposite (of a category)

\DeclareSymbolFont{symbolsC}{U}{txsyc}{m}{n}
\SetSymbolFont{symbolsC}{bold}{U}{txsyc}{bx}{n}
\DeclareFontSubstitution{U}{txsyc}{m}{n}
\DeclareMathSymbol{\Perp}{\mathrel}{symbolsC}{121} % commutation as a relation, originally from the txfonts package

\newcommand{\Sets}{\mathsf{Set}}
\newcommand{\CHs}{\mathsf{CHaus}}
\newcommand{\CGHs}{\mathsf{CGHaus}} % compactly generated Hausdorff spaces
\newcommand{\Calg}{\mathsf{C^*alg}_1} % unital C*-algebras
\newcommand{\cCalg}{\mathsf{cC^*alg}_1} % commutative unital C*-algebras
\newcommand{\pCalg}{\mathsf{pC^*alg}_1} % piecewise (unital) C*-algebras
\newcommand{\aCalg}{\mathsf{aC^*alg}_1} % piecewise (unital) C*-algebras with self-action
\newcommand{\Grp}{\mathsf{Grp}}
\newcommand{\pGrp}{\mathsf{pGrp}} % piecewise groups
\newcommand{\aGrp}{\mathsf{aGrp}} % piecewise groups with self-action

\title{(Almost) C*-algebras as sheaves with self-action}

\author{Cecilia Flori}
\author{Tobias Fritz}

\date{\today}

\address{Perimeter Institute for Theoretical Physics, Waterloo, Canada\vspace{-6pt}}
\address{Imperial College, London}
\email{ceciliaflori@imperial.ac.uk}

\address{Perimeter Institute for Theoretical Physics, Waterloo, Canada\vspace{-6pt}}
\address{Max Planck Institute for Mathematics in the Sciences, Leipzig, Germany}
\email{fritz@mis.mpg.de}

\keywords{Axiomatics of C*-algebras; sheaf theory; algebraic quantum mechanics; topos quantum theory}
\subjclass[2010]{Primary: 46L05 (General theory of C*-algebras), 46L60 (Applications of selfadjoint operator algebras to physics). Secondary: 18F20 (Presheaves and sheaves), 20A05 (Group theory, axiomatics and elementary properties).}

\thanks{\textit{Acknowledgements.} We thank Benno van den Berg, Chris Heunen and Manuel Reyes for discussion and crucial comments on a draft; Ryszard Kostecki, Klaas Landsman, Markus M\"uller, Sam Staton, Andreas Thom and Bas Westerbaan for further discussion; and Tom Leinster for pointing out Isbell's results on codensity. Research at Perimeter Institute is supported by the Government of Canada through Industry Canada and by the Province of Ontario through the Ministry of Economic Development and Innovation. During his time at Perimeter Institute, the second author has been supported by the John Templeton Foundation.}

\begin{document}

\begin{abstract}
Via Gelfand duality, a unital C*-algebra $A$ induces a functor from compact Hausdorff spaces to sets, $\CHs\to\Sets$. We show how this functor encodes standard functional calculus in $A$ as well as its multivariate generalization. Certain sheaf conditions satisfied by this functor provide a further generalization of functional calculus. Considering such sheaves $\CHs\to\Sets$ abstractly, we prove that the \emph{piecewise C*-algebras} of van den Berg and Heunen are equivalent to a full subcategory of the category of sheaves, where a simple additional constraint characterizes the objects in the subcategory. It is open whether  this additional constraint holds automatically, in which case piecewise C*-algebras would be the same as sheaves $\CHs\to\Sets$.

Intuitively, these structures capture the commutative aspects of C*-algebra theory. In order to find a complete reaxiomatization of unital C*-algebras within this language, we introduce \emph{almost C*-algebras} as piecewise C*-algebras equipped with a notion of inner automorphisms in terms of a \emph{self-action}. We provide some evidence for the conjecture that the forgetful functor from unital C*-algebras to almost C*-algebras is fully faithful, and ask whether it is an equivalence of categories. We also develop an analogous notion of \emph{almost group}, and prove that the forgetful functor from groups to almost groups is \emph{not} full.

In terms of quantum physics, our work can be seen as an attempt at a reconstruction of quantum theory from physically meaningful axioms, as realized by Hardy and others in a different framework. Our ideas are inspired by and also provide new input for the topos-theoretic approach to quantum theory.
\end{abstract}

\newgeometry{top=2cm}
\maketitle

\tableofcontents

\restoregeometry
\newpage
\section{Introduction}

C*-algebra theory is a blend of algebra and analysis which turns out to be much more than the sum of its parts, as illustrated by its fundamental results of Gelfand duality and the GNS representation theorem. Nevertheless, the C*-algebra axioms seem somewhat mysterious, and it may not be very clear what they mean or where they actually `come from'. To see the point, consider the axioms of groups for comparison: these have a clear meaning in terms of symmetries and the composition of symmetries, and this provides adequate motivation for these axioms. Do C*-algebras also have an interpretation which motivates their axioms in a similar manner?

A plausible answer to this question would be in terms of applications of C*-algebras to areas outside of pure mathematics. The most evident application of C*-algebras is to quantum mechanics and quantum field theory~\cite{strocchi,landsman,haag}. However, also in this context, the C*-algebra axioms do not seem well-motivated. In fact, not even the multiplication, which results in the algebra structure, does have a clear physical meaning. This is in stark contrast to other physical theories, such as relativity, where the mathematical structures that come up are derived from physical considerations and principles, often via the use of thought experiments. A similar derivation of C*-algebraic quantum mechanics does not seem to be known.

For these reasons, it seems pertinent to try and reformulate the C*-algebra axioms in a more satisfactory manner that would allow for a clear interpretation. This was our motivation for developing the notions of this paper.

For technical convenience, our C*-algebras are assumed unital throughout.

\subsection*{Summary and structure of this paper}

We start the technical development in Section~\ref{Calgsasfunctors} by assigning to every C*-algebra $A\in\Calg$ the functor $\CHs\to\Sets$ induced via the Yoneda embedding and Gelfand duality. It takes a compact Hausdorff space $X$ and maps it to the set of $*$-homomorphisms $C(X)\to A$. In terms of quantum mechanics, this is the set of projective measurements with outcomes in $X$, while the functoriality corresponds to post-processings or coarse-grainings of these measurements. We explain how this functor captures functional calculus for (commuting tuples of) normal elements of $A$, and how this encodes the `commutative aspect' of the structure of $A$. The physical interpretation is in terms of measurements with values in $X$ on the level of objects and post-processings between these on the level of morphisms.

Section~\ref{Calgsheaf} investigates which properties distinguish these functors from arbitrary functors $\CHs\to\Sets$. These properties take the form of sheaf conditions. Starting with the commutative case, we consider sheaf conditions satisfied by all hom-functors $\CHs(W,-):\CHs\to\Sets$. These can be interpreted in a manner similar to a conventional sheaf condition: while we think of the latter as identifying functions on a space with consistent assignments of values to all points, we now identify points with consistent assignments of values to all functions (Lemma~\ref{squarecover}). We then move on to consider sheaf conditions satisfied by the functors $\CHs\to\Sets$ associated to arbitrary $A\in\Calg$. Roughly, the question is how to `guarantee commutativity': under what conditions is a colimit of commutative C*-algebras itself commutative? We introduce \emph{directed cones} as a class of colimits that satisfy this, so that every C*-algebra becomes a functor $\CHs\to\Sets$ that satisfies the sheaf condition on all directed cones. The resulting category of sheaves $\Sh(\CHs)$ does not seem to be a category of sheaves on a (large) site since the directed cones do not form a coverage (Proposition~\ref{nocoverage}). Nevertheless, we show that $\Sh(\CHs)$ is at least locally small (Corollary~\ref{locsmall}) and well-powered (Proposition~\ref{wellpowered}). Furthermore, Lemma~\ref{replem} is a key technical result on the representability of our sheaves, which can be understood as a new characterization of commutative C*-algebras.

In Section~\ref{piecewisesec}, we relate our sheaves $\CHs\to\Sets$ to the \emph{piecewise C*-algebras} of van den Berg and Heunen~\cite{piecewise} (originally called \emph{partial C*-algebras}). The main result is Theorem~\ref{pCalgthm}, which identifies piecewise C*-algebras with a full subcategory of $\Sh(\CHs)$, the objects of which are characterized in terms of a simple additional condition. Since we do not know of any sheaf that would not satisfy this condition, $\Sh(\CHs)$ may even be equivalent to the category of piecewise C*-algebras (Problem~\ref{pCalgprob}).

Section~\ref{secsaC} asks which additional structure a piecewise C*-algebra (or suitable sheaf $\CHs\to\Sets$) could be equipped with such as to recover the noncommutative structure of a C*-algebra as well, i.e.~to obtain an equivalence with the category of C*-algebras. Our proposal is to consider the additional structure of a \emph{self-action}, in the sense of a notion of inner automorphisms: every unitary element should give rise to an automorphism, and these automorphisms should satisfy suitable conditions on commuting unitaries. Introducing such a self-action is motivated by the physical interpretation: it is one of the essential features of quantum mechanics that real-valued observables generate one-parameter families of inner automorphisms, by first exponentiating to a unitary (functional calculus) and then conjugating by that unitary (self-action). In this way, we obtain the category of \emph{almost C*-algebras} $\aCalg$, and we ask whether the forgetful functor $\Calg\to\aCalg$ is an equivalence. While it is clearly faithful, Theorem~\ref{Wfull} shows that it is also full on morphisms out of W*-algebras.

In order to understand better whether the forgetful functor $\aCalg\to\Calg$ could indeed be an equivalence, Section~\ref{secgrps} investigates an analogous question for groups instead of C*-algebras. We ask whether the forgetful functor $\Grp\to\aGrp$ from the category of groups to the category of \emph{almost groups} is an equivalence. Theorem~\ref{aGnotfull} shows that this is not the case, since the functor is not full on morphisms out of a free group.

\subsection*{How about other kinds of operator algebras?}

We expect that many of the ideas developed in this paper apply \emph{mutatis mutandis} to other kinds of operator algebras as well, and in particular to W*-algebras. In this sense, focusing on C*-algebras has been a somewhat arbitrary choice made in the present work. In fact, as indicated by Lemma~\ref{weakvalulem} and especially by Theorem~\ref{Wfull}, the W*-algebra case allows for the derivation of more powerful results than we have been able to prove in the C*-algebra setting. The main reason for us to treat the C*-algebra case in this paper is the greater technical simplicity of topology over measure theory. For example, a W*-algebra version of Gelfand duality in terms of an equivalence of the category of commutative W*-algebras with a suitable category of measurable spaces is not readily available in the literature.

\subsection*{Relation to topos quantum theory} 

The present ideas have commonalities with and were partly inspired by the topos-theoretic approach to quantum physics~\cite{thing,cecilia,hls}. Nevertheless, there are important differences, which may also provide a new direction for topos quantum theory. Crucially, topos quantum theory is formulated in terms of a topos that depends on the particular physical system under consideration, namely the category of presheaves on the poset of commutative subalgebras of the algebra of observables $A$. Instead of working with commutative subalgebras only, we consider \emph{all} $*$-homomorphisms $C(X)\to A$ for \emph{all} commutative C*-algebras $C(X)$. Doing so means that $A$ becomes a functor $\CHs\to\Sets$. In this way, we can consider all physical systems as described on the same footing as objects in the functor category $\Sets^\CHs$ or the category of sheaves $\Sh(\CHs)$.

\subsection*{Notation and terminology}

For us, `C*-algebra' always means `unital C*-algebra'. Likewise, our $*$-homomorphisms are always assumed to be unital, unless noted otherwise (as in the proof of Theorem~\ref{Wfull}). This already applies to the following index of our notation, which lists the conventions for our most commonly used mathematical symbols:

\newcommand{\notation}[2]{#1: & \textrm{\: #2}. \\}
\allowdisplaybreaks

\begin{align*}
\notation{W,X,Y,Z}{compact Hausdorff spaces}
\notation{\mathbf{1},\ldots,\mathbf{4}}{A compact Hausdorff on the corresponding number of points, where we write e.g.~$\mathbf{4}=\{0,1,2,3\}$}
\notation{w,x,y,z}{points in a compact Hausdorff space}
\notation{f,g,h,k}{continuous functions between compact Hausdorff spaces}
\notation{\Box,\bigcirc,\Tl}{unit square, unit disk and unit circle, considered as compact subsets of $\Cl$}
\notation{A,B}{C*-algebras or piecewise C*-algebras (Definition~\ref{piecewisedef})}
\notation{M_n}{The C*-algebra of $n\times n$ matrices with entries in $\Cl$}
\notation{\alpha,\beta,\gamma,\nu,\tau}{normal elements in a C*-algebra, or (more generally) $*$-homomorphisms of the type $C(X)\to A$}
\notation{\zeta}{a $*$-homomorphism or piecewise $*$-homomorphism of the type $A\to B$}
\notation{\mathfrak{a},\mathfrak{b}}{self-action of a piecewise C*-algebra (Definition~\ref{almostdef}) or a piecewise group (Definition~\ref{almostgroupdef})}
\end{align*}

The normal part of a C*-algebra $A$ is
\[
\Cl(A) \defin \{\: \alpha\in A \:|\: \alpha\alpha^* = \alpha^*\alpha \:\}.
\]
We also think of it as the set of `$A$-points' of $\Cl$. More generally, for $A\in\Calg$ and a closed subset $S\subset\Cl$, we also write
\[
S(A) \defin \{\: \alpha\in \Cl(A) \:|\: \spc(\alpha)\subseteq S \:\}
\]
for the set of normal elements with spectrum in $S$, and similarly $S(\zeta):S(A)\to S(B)$ for the resulting action of a $*$-homomorphism $\zeta:A\to B$ on these elements. For example, $\Rl(A)$ denotes the self-adjoint part of a C*-algebra, and similarly $\Tl(A)$ is the unitary group. This sort of notation may be familiar from algebraic geometry, where for a ring $A$, the set of $A$-points of a scheme $S$ is denoted $S(A)$. We also use the standard notation $C(X)$ for the $\Cl$-valued continuous functions on a space $X$. Unfortunately, this is very similar notation despite being different in nature.

We work with the following categories:

\begin{align*}
\notation{\CHs}{compact Hausdorff spaces with continuous maps}
\notation{\CGHs}{compactly generated Hausdorff spaces with continuous maps}
\notation{\Calg}{C*-algebras with $*$-homomorphisms}
\notation{\cCalg}{commutative C*-algebras with $*$-homomorphisms}
\notation{\pCalg}{piecewise C*-algebras (Definition~\ref{piecewisedef}) with piecewise $*$-homomorphisms (Definition~\ref{phomdef})}
\notation{\aCalg}{almost C*-algebras (Definition~\ref{almostdef}) with almost $*$-homomorphisms (Definition~\ref{ahomdef})}
\notation{\Grp}{groups with group homomorphisms}
\notation{\pGrp}{piecewise groups (Definition~\ref{piecewisegroupdef}) with piecewise group homomorphisms (Definition~\ref{pghomdef})}
\notation{\aGrp}{almost groups (Definition~\ref{almostgroupdef}) with almost group homomorphisms (Definition~\ref{aghomdef})}
\end{align*}

Throughout, all diagrams are commutative diagrams, unless explicitly stated otherwise.

\newpage

\section{C*-algebras as functors $\CHs\to\Sets$}
\label{Calgsasfunctors}

In this section, we explain how to regard a C*-algebra as a functor $\CHs\to\Sets$, and how this encodes the usual functional calculus for normal elements in a C*-algebra, as well as its multivariate generalization.

The Yoneda embedding realizes a C*-algebra $A$ as the hom-functor
\[
\Calg(-,A) \: :\: \Calg^\op \to \Sets.
\]
We are interested in studying this hom-functor on the commutative C*-algebras, meaning that we consider its restriction to a functor $\cCalg^\op\to\Sets$. Applying Gelfand duality, we can equivalently consider it as a functor
\[
-(A) \: :\: \CHs \to \Sets,
\]
assigning to every compact Hausdorff space $X\in\CHs$ a set $X(A)$, which is the set of all $*$-homomorphisms $C(X)\to A$. Our notation $X(A)$ suggests thinking of it as the set of \emph{generalized $A$-points} of $X$.

\begin{example}
If $X$ is finite, a $*$-homomorphism $C(X)\to A$ or generalized $A$-point in $X$ corresponds to a \emph{partition of unity} in $A$ indexed by $X$, i.e.~a family of pairwise orthogonal projections summing up to $1$.
\end{example}

\begin{example}
If $A$ is a W*-algebra, the spectral theorem~\cite[Theorem~1.44]{folland} implies that $X(A)$ is precisely the collection of all regular projection-valued measures on $X$ with values in $A$.
\end{example}

\begin{remark}
\label{measurements}
In terms of algebraic quantum mechanics, where a physical system is described by a C*-algebra $A$ of observables~\cite{strocchi,landsman}, we interpret a $*$-homomorphism $\alpha : C(X)\to A$ as a projective measurement with values in $X$, described in the Heisenberg picture. So the physical meaning of our $X(A)$ is as the collection of all measurements with outcomes in the space $X$.
\end{remark}

\begin{remark}
Those $*$-homomorphisms $C(X)\to A$ whose image is in the center of $A$ are called \emph{$C(X)$-algebras}, and they correspond exactly to upper semicontinuous C*-bundles over $X$~\cite{nilsen}\footnote{We thank Klaas Landsman for pointing this out to us.}.
\end{remark}

At the level of morphisms, every $f:X\to Y$ acts by composing a $*$-homomorphism $\alpha : C(X)\to A$ with $C(f)$ to $\alpha\circ C(f) : C(Y)\to A$, so that
\begin{align}
\begin{split}
\label{faction}
f(A) :  X(A) 	& \longrightarrow Y(A) \\
	\alpha 	& \longmapsto \alpha\circ C(f)
\end{split}
\end{align}
is the action of $f$ on generalized $A$-points.

\begin{remark}
\label{postproc}
The physical interpretation of $f(A)$ is as a \emph{post-processing} or \emph{coarse-graining} of measurements. Under $f(A)$, a measurement $\alpha : C(X)\to A$ with values in $X$ becomes a measurement $\alpha\circ C(f):C(Y)\to A$ with values in $Y$, implemented by first conducting the original measurement $\alpha$ and then processing the outcome via application of the function $f$. Since we work in the Heisenberg picture, the order of composition is reversed, so that $C(f)$ happens first.
\end{remark}

This construction is also functorial in $A$: for any $*$-homomorphism $\zeta:A\to B$ and $X\in\CHs$, we have $X(\zeta) : X(A) \to X(B)$. Furthermore, for any $f:X\to Y$ there is the evident naturality diagram
\[
\xymatrix{	X(A) \ar[d]_{X(\zeta)} \ar[r]^{f(A)}	& Y(A) \ar[d]^{Y(\zeta)}	\\
		X(B) \ar[r]_{f(B)}			& Y(B)				}
\]
which expresses the bifunctoriality of the hom-functor $\Calg(-,-)$ in our setup.

Before proceeding with technical developments in the next section, it is worthwhile pondering how these considerations relate to functional calculus.

\subsection*{Functoriality captures the `commutative part' of the C*-algebra structure}
\label{funcalc}

In a somewhat informal sense, the functor $-(A)$ captures the entire `commutative part' of the structure of a C*-algebra $A$. We will obtain a precise result along these lines as Theorem~\ref{pCalgthm}. Here, we perform some simple preparations.

\begin{lemma}
For any compact set $S\subseteq\Cl$, evaluating an $\alpha : C(S)\to A$ on $\id_S : S\to\Cl$,
\be
\label{evalid}
\alpha\longmapsto\alpha(\id_S),
\ee
is a bijection between $S(A)$ and the normal elements of $A$ with spectrum in $S$.
\label{normalcorrespond}
\end{lemma}

\begin{proof}
If $\alpha,\beta:C(S)\to A$ coincide on $\id_S$, then they must coincide on the *-algebra generated by $\id_S$. Since $\id_S$ separates points, this *-algebra is dense in $C(S)$ by the Stone-Weierstrass theorem, so that $\alpha=\beta$ by continuity. This establishes injectivity of~\eqref{evalid}.

Concerning surjectivity, applying functional calculus to a given normal element with spectrum in $S$ results in a $*$-homomorphism $C(S)\to A$ which realizes the given element via~\eqref{evalid}.
\end{proof}

Due to this correspondence, we will not distinguish notationally between a $*$-homomorphism $\alpha : C(S)\to A$ and its associated normal element, i.e.~we also denote the latter simply by $\alpha\in A$. Moreover, we can also think of a $*$-homomorphism $C(X)\to A$ for arbitrary $X\in\CHs$ as a sort of `generalized normal element' of $A$.

For any two compact $S,T\subseteq\Cl$ and $f:S\to T$, functional calculus---in the sense of applying $f$ to normal elements with spectrum in $S$---is encoded in two ways:
\begin{itemize}
\item in evaluating an $\alpha : C(S)\to A$ on $f:S\to\Cl$, as in the proof of Lemma~\ref{normalcorrespond};
\item in the functoriality $f(A) : S(A)\to T(A)$, since applying this functorial action to $\alpha$ results in the same normal element of $A$,
\be
\label{fundrel}
f(A)(\alpha)(\id_T) \stackrel{\eqref{faction}}{=} (\alpha\circ C(f))(\id_T) = \alpha(C(f)(\id_T)) = \alpha(\id_T\circ f) = \alpha(f).
\ee
\end{itemize}
From now on, what we mean by `functional calculus' is the functoriality, i.e.~the second formulation.

Writing $\bigcirc\subseteq\Cl$ for the unit disk, the normal elements of norm $\leq 1$ are identified with the $*$-homomorphisms $\alpha : C(\bigcirc)\to A$. For every $r\in [0,1]$, we have the multiplication map $r\cdot : \bigcirc\to\bigcirc$, so that $(r\cdot)(A) : \bigcirc(A) \to \bigcirc(A)$ represents scalar multiplication of normal elements by $r$. Based on this, we can recover the norm of a normal element $\alpha\in\bigcirc(A)$ as the largest $r$ for which $\alpha$ factors through $C(r)$,
\[
||\alpha|| = \max\: \{\: r\in[0,1] \:|\: \alpha\in\im((r\cdot)(A)) \:\}\:.
\]

As we will see next, the functoriality also captures part of the binary operations of a C*-algebra.

\begin{lemma}
\label{pairscorrespond}
For $S,T\subseteq\Cl$, applying functoriality to the product projections 
\be
\label{prodprojs}
p_S \: : \: S\times T\longrightarrow S,\qquad p_T \: : \: S\times T\longrightarrow T
\ee
establishes a bijection between $(S\times T)(A)$ and pairs of \emph{commuting} normal elements $(\alpha,\beta)\in A\times A$ with $\spc(\alpha)\subseteq S$ and $\spc(\beta)\subseteq T$.
\end{lemma}

This generalizes Lemma~\ref{normalcorrespond} to commuting pairs of normal elements. Of course, there are analogous statements for tuples of any size (finite or even infinite), and this encodes multivariate functional calculus.

\begin{proof}
We need to show that the map
\[
(p_S(A),p_T(A)) \: : \: (S\times T)(A) \longrightarrow S(A)\times T(A)
\]
is injective, and that its image consists of precisely the pairs $(\alpha,\beta)$ with $\alpha:C(S)\to A$ and $\beta:C(T)\to A$ that have commuting ranges. Injectivity holds because $p_S : S\times T\to\Cl$ and $p_T:S\times T\to\Cl$ separate points, so that the same argument as in the proof of Lemma~\ref{normalcorrespond} applies. For surjectivity, let $\alpha$ and $\beta$ be given. Since their ranges commute, we can find a commutative subalgebra $C(X)\subseteq A$ that contains both, so that the pair $(\alpha,\beta)$ has a preimage in the upper right corner of the diagram
\[
\xymatrix{	(S\times T)(C(X)) \ar[r] \ar[d]	& S(C(X))\times T(C(X)) \ar[d]	\\
		(S\times T)(A) \ar[r]		& S(A)\times T(A)		}
\]
Now the upper row is equal to the canonical map $\CHs(X,S\times T)\to \CHs(X,S)\times\CHs(X,T)$, which is a bijection due to the universal property of $S\times T$. Hence we can find a preimage of $(\alpha,\beta)$ also in the upper left corner, and then also in the lower left corner by commutativity of the diagram.
\end{proof}

\begin{remark}
In the physical interpretation, the elements of $(S\times T)(A)$ are measurements that have outcomes in $S\times T$ (Remark~\ref{measurements}). Lemma~\ref{pairscorrespond} now shows that such a measurement corresponds to a pair of \emph{compatible} measurements taking values in $S$ and $T$, respectively, and one obtains these measurements by coarse-graining along the product projections~\eqref{prodprojs}, i.e.~by forgetting the other outcome.
\end{remark}

As part of bivariate functional calculus, we can now consider the addition map
\be
\label{addition}
S\times T \longrightarrow S + T,\qquad (x,y)\longmapsto x + y,
\ee
where $S+T$ is the Minkowski sum
\[
S + T = \{\: x + y \:|\: x\in S,\: y\in T\:\},
\]
again considered as a compact subset of $\Cl$. Under the identifications of Lemmas~\ref{normalcorrespond} and~\ref{pairscorrespond}, the addition map
\be
\label{additionA}
+(A) \: :\: (S\times T)(A) \longrightarrow (S + T)(A).
\ee
takes a pair of commuting normal elements with spectra in $S$ and $T$ and takes it to a normal element with spectrum in $S+T$.

\begin{lemma}
On commuting normal elements, this recovers the usual addition in $A$.
\end{lemma}

\begin{proof}
By Lemma~\ref{pairscorrespond}, it is enough to take a $\gamma\in (S\times T)(A)$ and to compute the resulting normal element that one obtains by applying $+(A)$ in a manner analogous to~\eqref{fundrel},
\begin{align*}
(+(A))(\gamma)(\id_{S+T}) 	& \stackrel{\eqref{faction}}{=} (\gamma \circ C(+))(\id_{S+T}) = \gamma(\id_{S+T}\circ +) \\
				& = \gamma(\id_S \circ p_S + \id_T\circ p_T) \\
				& = \gamma(\id_S\circ p_S) + \gamma(\id_T\circ p_T) \\
				& = (\gamma\circ C(p_S))(\id_S) + (\gamma\circ C(p_T))(\id_T) \\
				& \stackrel{\eqref{faction}}{=} (p_S(A))(\gamma)(\id_S) + (p_T(A))(\gamma)(\id_T),
\end{align*}
where the crucial assumption of additivity of $\gamma$ has been used to obtain the expression in the third line.
\end{proof}

In the analogous manner, one can show that the multiplication map
\be
\label{multiplication}
S\times T \longrightarrow ST,\qquad (x,y)\longmapsto xy.
\ee
lets us recover the product of two commuting normal elements in $A$. More generally, we can recover any polynomial or continuous function of any number of commuting normal elements.

In summary, we think of the functor $-(A):\CHs\to\Sets$ associated to $A\in\Calg$ as a generalization of functional calculus, which remembers the entire `commutative structure' of $A$. The generalization is from applying functions to individual normal elements---as in the conventional picture of functional calculus---to applying functions to `generalized' normal elements in the guise of $*$-homomorphisms of the form $C(X)\to A$. In particular, the C*-algebra operations acting on commuting normal elements are encoded in the functoriality. In the remainder of this paper, we will always have this point of view in mind, together with its physical interpretation:
\[
\text{functoriality = generalized functional calculus = post-processing of measurements.}
\]

\begin{remark}
\label{reconstruct}
In Section~\ref{guarcommsec}, we will also consider functors $F:\CHs\to\Sets$ that do not necessarily arise from a C*-algebra in this way. In terms of the physical interpretation, this means that we attempt to model physical systems not in terms of their algebras of observables as the primary structure, but in terms of a functor $F$ as the most fundamental structure that describes physics. This is motivated by the fact that the C*-algebra structure of the observables is (a priori) not physically well-motivated, as discussed in the introduction. Thanks to Remarks~\ref{measurements} and~\ref{postproc}, our functors $F:\CHs\to\Sets$ do have a meaningful operational interpretation in terms of measurements: $F(X)$ is the set of (projective) measurements with outcomes in $X$, and the action of $F$ on morphisms is the post-processing. This bare-bones structure turns out to carry a surprising amount of information about the algebra of observables. We will try to equip $F$ with additional properties and structure such as to uniquely specify the algebra of observables.

In spirit, this approach is similar to the existing reconstructions of quantum mechanics from operational axioms~\cite{grinbaum}. In recent years, a wide range of reconstruction theorems with a large variety of choices for the axioms have been derived, as pioneered by Hardy~\cite{hardy1,hardy2}. In these theorems, `quantum mechanics' refers to the Hilbert space formulation in finite dimensions, and the reconstruction theorems recover the Hilbert space structure within the framework of general probabilistic theories. In contrast to this, our work focuses on the C*-algebraic formulation of quantum mechanics and is not limited to a finite-dimensional setting. Also, we do not make use of the possibility of taking stochastic mixtures: since we are (currently) only dealing with projective measurements, taking stochastic mixtures is not even possible in our setup.
\end{remark}

\newpage
\section{C*-algebras as sheaves $\CHs\to\Sets$}
\label{Calgsheaf}

Functional calculus lets us apply functions to operators, or more generally to $*$-homomorphisms $C(X)\to A$ as in the previous section. In some situations, one can also go the other way: for certain families of functions $\{f_i : X\to Y_i\}_{i\in I}$ with common domain, a collection of $*$-homomorphisms $\{\beta_i : C(Y_i)\to A\}_{i\in I}$ arises from a unique $*$-homomorphism $\alpha : C(X)\to A$ by functoriality along the $f_i$ if and only if the $\beta_i$ satisfy a simple compatibility requirement. This property is a \emph{sheaf condition}, and it turns our functors $-(A)$ into sheaves on the category $\CHs$.

\begin{remark}
We emphasize already at this point that the sheaf conditions that we consider do not arise from a Grothendieck topology (on $\CHs^\op$), since the axiom of stability under pullback fails to hold. Also, while sheaf conditions are typically formulated for contravariant functors (i.e.~presheaves), our sheaves live in a covariant setting. While we could speak of `cosheaves' to emphasize this distinction, this term usually refers to dualizing the standard notion of sheaf on the codomain category, while we dualize on the domain category.
\end{remark}

A good way of talking about sheaf conditions on large categories is not in terms of sieves or cosieves---which would usually have to be large---but in terms of cocones or cones~\cite{shulman}:

\begin{definition}
A \emph{cone} in $\CHs$ is any small family of morphisms $\{f_i :X\to Y_i\}_{i\in I}$ with common domain.
\end{definition}

\begin{definition}
A functor $F:\CHs\to\Sets$ satisfies the \emph{sheaf condition} on a cone $\{f_i : X\to Y_i\}_{i\in I}$ if the $F(f_i)$ implement a bijection between the sections $\alpha\in F(X)$ and the families of sections $\{\beta_i\}_{i\in I}$ with $\beta_i\in F(Y_i)$ that are \emph{compatible} in the following sense: for any $i,j\in I$ and any diagram
\be
\begin{split}
\label{compdiag}
\xymatrix{	X \ar[r]^{f_i} \ar[d]_{f_j}	& Y_i \ar[d]^g	\\
		Y_j \ar[r]_h			& Z		}
\end{split}
\ee
we have $F(g)(\beta_i) = F(h)(\beta_j)$.
\end{definition}

Since $\CHs$ has pushouts, the compatibility condition holds if and only if it holds on every pushout diagram
\[
\xymatrix{	X \ar[r]^{f_i} \ar[d]_{f_j}	& Y_i \ar[d]			\\
		Y_j \ar[r]			& Y_i\pushout{f_i}{f_j} Y_j	}
\]
Hence the sheaf condition holds on $\{f_i\}$ if and only if the diagram
\[
\xymatrix{	F(X) \ar[r]	& \mathlarger{\mathlarger{\prod}}_{i\in I} F(Y_i) \ar@<1ex>[r] \ar@<-1ex>[r]	& \mathlarger{\mathlarger{\prod}}_{i,j\in I} F(Y_i\pushout{f_i}{f_j} Y_j). }
\]
is an equalizer in $\Sets$, where the arrows are the canonical ones~\cite[p.~123]{MM}. At times it is convenient to apply the compatibility condition as in~\eqref{compdiag} instead of considering the pushout, while at other times it is necessary to work with the pushout explicitly.

\subsection*{Effective-monic cones in $\CHs$}

Since we are interested in sheaf conditions satisfied by a functor of the form $-(A):\CHs\to\Sets$ for $A\in\Calg$, it makes sense to consider the commutative case first. Then our functor takes the form $-(C(W))$, which is isomorphic to the hom-functor $\CHs(W,-)$.

\begin{definition}[{e.g.~\cite[Definition~2.22]{shulman}}]
\label{effmondef}
A cone $\{f_i:X\to Y_i\}_{i\in Y}$ in $\CHs$ is \emph{effective-monic} if every representable functor $\CHs(W,-)$ satisfies the sheaf condition on it.
\end{definition}

Hence $\{f_i\}$ is effective-monic if and only if $X$ is the equalizer in the diagram
\[
	\xymatrix{	X \ar[r]	& \mathop{\mathlarger{\mathlarger{\prod}}_{i\in I}} Y_i \ar@<1ex>[r] \ar@<-1ex>[r]	& \mathlarger{\mathlarger{\prod}}_{i,j\in I} (Y_i\pushout{f_i}{f_j}Y_j), }
\]
or equivalently the limit in the diagram
\be
\begin{split}
\label{pointssheaf}
\xymatrix{					&& Y_i \ar[dr]	& \vdots			\\
		X \ar[drr]_{f_j} \ar[urr]^{f_i}	&& \vdots	& Y_i\pushout{f_i}{f_j}Y_j	\\
						&& Y_j \ar[ur]	& \vdots			}
\end{split}
\ee

\begin{example}
\label{limit}
Let $\Lambda$ be a small category and $L:\Lambda\to\CHs$ a functor of which we consider the limit $\lim_\Lambda L\in\CHs$. The limit projections $p_\lambda : \lim_\Lambda L \to L(\lambda)$ assemble into a cone $\{p_\lambda\}_{\lambda\in\Lambda}$, which is effective-monic. 
\end{example}

Fortunately, it is not necessary to consider arbitrary $W$ in Definition~\ref{effmondef}:

\begin{lemma}
A cone $\{f_i\}$ is effective-monic if and only if $\CHs(\mathbf{1},-)$ satisfies the sheaf condition on it.
\label{pointsonly}
\end{lemma}

\begin{proof}
$\CHs$ is well-known to be monadic over $\Sets$, with the forgetful functor being precisely the functor of points $\CHs(\mathbf{1},-) : \CHs\to\Sets$. In particular, this functor creates limits.
\end{proof}

So in words, $X$ must be the subspace of the product space $\prod_{i\in I} Y_i$ consisting of all those families of points $\{y_i\}_{i\in I}$ such that the image of $y_i\in Y_i$ coincides with the image of $y_j\in Y_j$ in the pushout space $Y_i\pushout{f_i}{f_j}Y_j$. This condition also applies for $j=i$, in which case it is equivalent to $y_i\in\im(f_i)$.

\begin{remark}
For a given $Y$, the cone of all functions $\{f:X\to Y\}_{f:X\to Y}$ is effective-monic for every $X$ if and only if $Y$ is codense.
\label{codense}
\end{remark}

While these categorical considerations have been extremely general, we now get into the specifics of $\CHs$. We write $\Box\defin[0,1]\times[0,1]$ for the unit square, and consider it as embedded in $\Box\subseteq\Rl^2=\Cl$, where the unit interval $[0,1]\subseteq\Rl$ is an edge of $\Box$.

\begin{lemma}
\label{squarecover}
For every $X\in\CHs$, the cone $\{f:X\to\Box\}_{f:X\to\Box}$ consisting of all functions $f:X\to\Box$ is effective-monic.
\end{lemma}

By Remark~\ref{codense}, this is a restatement of the known fact that $\Box$ is codense in $\CHs$~\cite{isbell}.

While one thinks of a conventional sheaf condition as saying that a function is uniquely determined by a compatible assignment of values to all (local neighbourhoods of) points, this sheaf condition says that a point is uniquely determined by a compatible assignment of values to all functions.

\begin{proof}
We need to show that the diagram
\[
\xymatrix{ X \ar[r]	& \mathlarger{\mathlarger{\prod}}_{f:X\to \Box} \Box \ar@<1ex>[r] \ar@<-1ex>[r]	& \mathlarger{\mathlarger{\prod}}_{g,h:X\to\Box} (\Box \pushout{g}{h}\Box)	}
\]
is an equalizer. Since functions $X\to\Box$ separate points in $X$, it is clear that the map $X\to\prod_f\Box$ is injective.

Surjectivity is more difficult. Suppose that $v \in \prod_{f:X\to\Box} \Box$ is a compatible family of sections. Then in particular, we have 
\be
v(hf) = h(v(f)) \quad\textrm{ for all }\quad h:\Box\to \Box
\label{noncontextuality}
\ee
as an instance of the compatibility condition, since the square
\be
\begin{split}
\label{weakcomp}
\xymatrix{	X \ar[r]^{hf} \ar[d]_f	& \Box \ar@{=}[d]	\\
		\Box \ar[r]_h		& \Box			}
\end{split}
\ee
commutes.

We have to show that there exists a point $x\in X$ with $v(f) = f(x)$ for all $f:X\to\Box$. This set of equations is equivalent to $x\in \bigcap_f f^{-1}(v(f))$. Hence it is enough to show that $\bigcap_f f^{-1}(v(f))$ is nonempty. By compactness, it is sufficient to prove that any finite intersection
\[
f_1^{-1}(v(f_1)) \cap \ldots \cap f_n^{-1}(v(f_n))
\]
for a finite set of functions $f_1,\ldots,f_n:X\to\Box$ is nonempty. Using induction on $n$, the induction step is obvious if for given $f_1,f_2$ we can exhibit $g:X\to\Box$ such that
\[
g^{-1}(v(g)) = f_1^{-1}(v(f_1)) \cap f_2^{-1}(v(f_2)).
\]
First, by~\eqref{noncontextuality}, we can assume that both $f_1$ and $f_2$ actually take values in $[0,1]$, e.g.~by considering
\[
h_1 \: : \: \Box\longrightarrow[0,1], \qquad t\longmapsto |t - v(f_1)|
\]
and replacing $f_1$ by $h_1 f_1$, which results in
\[
(h_1 f_1)^{-1}(v(h_1 f_1)) \stackrel{\eqref{noncontextuality}}{=} f_1^{-1}(h_1^{-1}(h_1(v(f_1)))) = f_1^{-1}(h_1^{-1}(0)) = f_1^{-1}(v(f_1)),
\]
and similarly for $f_2$. After this replacement, we can take $g(t)\defin (f_1(t),f_2(t))$, and the induction step is complete upon applying~\eqref{noncontextuality} to the two coordinate projections.

Finally, we need to show that any individual set $f^{-1}(v(f))$ is nonempty as the base of the induction. For given $s\in[0,1]\setminus\im(f)$, choose $h$ such that $h(\im(f))=\{0\}$ and $h(s)=1$ by the Tietze extension theorem. Then
\[
0 = v(0) = v(hf) = h(v(f)),
\]
and hence $v(f)\neq s$. Therefore $v(f)\in\im(f)$, as was to be shown.
\end{proof}

For us, this effective-monic cone is the most important one. We now consider some other examples of effective-monic cones in $\CHs$, which shed some light on their general behaviour. This is relevant for our main line of thought only as a source of examples.

As the counterexample given in the proof of~\cite[Theorem 2.6]{isbell} shows, this does generally not hold with $[0,1]$ in place of $\Box$. However, at least if $X$ is extremally disconnected, then it is still true, as an immediate consequence of the following result:

\begin{lemma}
\label{weakvalulem}
If $X$ is extremally disconnected, then $\{f:X\to\mathbf{4}\}$ is effective-monic.
\end{lemma}

Here, we write $\mathbf{4}\defin\{0,1,2,3\}$, and the proof uses indicator functions $\chi_Y : X\to\mathbf{4}$ of clopen sets $Y\subseteq X$.

\begin{proof}
Since the clopen sets separate points, the injectivity is again clear and the burden of the proof is in the surjectivity. So let $v:\mathbf{4}^X\to\mathbf{4}$ be a compatible family of sections.

As in the proof of Lemma~\ref{squarecover}, we show that the intersection 
\[
\bigcap_{Y \text{ clopen, } v(\chi_Y) = 1} Y
\]
is nonempty. Again by compactness and an induction argument as in the proof of Lemma~\ref{squarecover}, it is enough to show that for any clopen $Y_1,Y_2\subseteq X$ with $v(\chi_{Y_1}) = 1$ and $v(\chi_{Y_2}) = 1$, we also have $v(\chi_{Y_1\cap Y_2}) = 1$. To see this, we consider the function
\[
f\defin \chi_{Y_1} + 2\chi_{Y_2},
\]
and apply the compatibility condition in the form~\eqref{noncontextuality} for various $h$. Choosing $h$ such that $0,2\mapsto 0$ and $1,3\mapsto 1$ results in $hf=\chi_{Y_1}$, and hence $v(f)\in \{1,3\}$. Similarly, mapping $0,1\mapsto 0$ and $2,3\mapsto 1$ yields $hf=\chi_{Y_2}$, and therefore $v(f)\in\{2,3\}$. Overall, we obtain $v(f)=3$, and apply $h$ with $0,1,2\mapsto 0$ and $3\mapsto 1$ to conclude $v(\chi_{Y_1\cap Y_2})=1$ from $hf=\chi_{Y_1\cap Y_2}$.

So there is at least one point $x_0\in X$ such that $v(\chi_Y)=1$ implies $x_0\in Y$ for all clopen $Y\subseteq X$. We then claim that $v(f) = f(x_0)$ for all $f:X\to\mathbf{4}$. This follows from writing
\[
f = 0\chi_{Y_0} + 1\chi_{Y_1} + 2\chi_{Y_2} + 3\chi_{Y_3}
\]
for a partition of $X$ by clopens $Y_0,Y_1,Y_2,Y_3\subseteq X$, and applying~\eqref{noncontextuality} with $h$ such that $v(f)\mapsto 1$, while the other three integers map to $0$. 
\end{proof}

A singleton cone $\{f:X\to Y\}$ is effective-monic if and only if $f$ is injective. For cones consisting of exactly two functions, the necessary and sufficient criterion is as follows:

\begin{lemma}
\label{malcev}
A cone $\{f:X\to Y,g:X\to Z\}$ consisting of exactly two functions is effective-monic if and only if the pairing $(f,g) : X\to Y\times Z$ is a Mal'cev relation, meaning that $f$ and $g$ are jointly injective and their joint image
\[
R\defin\im((f,g))\subseteq Y\times Z
\]
satisfies the implication
\be
\label{eqmalcev}
(y,z) \in R,\qquad (y',z) \in R,\qquad (y,z') \in R \qquad\Longrightarrow\qquad (y',z') \in R.
\ee
\end{lemma}

For the notion of Mal'cev relation, see~\cite{garner}.

\begin{proof}
We use the criterion of Lemma~\ref{pointsonly}. The injectivity part of the sheaf condition is equivalent to injectivity of $(f,g) : X\to Y\times Z$. Assuming that this holds, we identify $X$ with the joint image $R\subseteq Y\times Z$. 

Now if $\{f,g\}$ is effective-monic and we have $y,y'\in Y$ and $z,z'\in Z$ as in~\eqref{eqmalcev}, then each of the three pairs $(y,z)$, $(y',z)$ and $(y,z')$ represents a point of $X$. So since $(y,z)$ is in particular a compatible pair of sections, in $Y\pushout{f}{g}Z$ the image of $y$ coincides with the image of $z$. By the same reasoning applied to $(y',z)$, also $y'$ maps to the same point in $Y\pushout{f}{g}Z$, and by $(y,z')$ so does $z'$. Hence also $(y',z')$ is a compatible pair of sections, which must correspond to a point of $X$ due to the sheaf condition.

Conversely, suppose that~\eqref{eqmalcev} holds. The pushout $Y\pushout{f}{g}Z$ is the quotient of the coproduct $Y\amalg Z$ by the closed equivalence relation generated by $f(x)\sim g(x)$ for all $x\in X$, i.e.~by $y\sim z$ for all $(y,z)\in R$. In terms of relational composition, it is straightforward to check that
\[
	\id_{Y\amalg Z} \cup R \cup R^\op \cup (R\circ R^\op) \cup (R^\op\circ R) 
\]
is already an equivalence relation thanks to~\eqref{eqmalcev}. As a finite union of closed sets, it is also closed, and hence two points in $Y\amalg Z$ get identified in $Y\pushout{f}{g} Z$ if and only if they satisfy this relation. In particular, $y\in Y$ and $z\in Z$ map to the same point in $Y\pushout{f}{g}Z$ if and only if $(y,z)\in R$.
\end{proof}

In general, the pushout of an effective-monic cone along an arbitrary function is not effective-monic again. The following example shows even that the effective-monic cones on $\CHs$ do not form a coverage; an even more drastic example can be found in the proof of Proposition~\ref{nocoverage}.

\begin{example}
\label{ex4to3}
Take $X\defin\mathbf{4}=\{0,1,2,3\}$, and consider two maps to spaces with 3 points, 
\[
f\: : \: \{0,1,2,3\} \longrightarrow \{01,2,3\},\qquad g\: : \: \{0,1,2,3\} \longrightarrow \{0,1,23\},
\]
as illustrated by the projection maps in Figure~\ref{fig4to3}. By Lemma~\ref{malcev}, this cone is effective-monic. However, taking the pushout along the identification map
\[
h \: : \: \{0,1,2,3\} \longrightarrow \{0,12,3\}
\]
results in a cone consisting of $f' : \{0,12,3\}\to \{012,3\}$ and $g':\{0,12,3\}\to \{0,123\}$. Since the criterion of Lemma~\ref{malcev} fails, the cone $\{f',g'\}$ is not effective-monic. In particular, the pushout of an effective-monic cone is not necessarily effective-monic again. Worse, the collection of all effective-monic cones is not a coverage: for our original $\{f,g\}$, there does not exist any effective-monic cone $\{k_i : \{0,12,3\}\to Y_i\}_{i\in I}$ such that every $k_i h$ would factor through $f$ or $g$,
\[
\xymatrix{							&&				&	\{01,2,3\} \ar@{-->}[dd]^?	\\
		\{0,1,2,3\} \ar[urrr]^f \ar[rr]_g \ar[d]_h	&& \{0,1,23\} \ar@{-->}[dr]_?						\\
		\{0,12,3\} \ar[rrr]_{k_i}			&&				&	Y_i				}
\]
The reason is as follows: for every $i\in I$, we would need to have $k_i(0) = k_i(12)$ or $k_i(12) = k_i(3)$. If the former happens, consider the point $y_i\defin k_i(3)\in Y_i$, while if the latter happens take $y_i\defin k_i(0)$. (If both cases apply, these two prescriptions result in the same point $y_i = k_i(0) = k_i(3)$.) It is easy to check that the resulting family of points $\{y_i\}_{i\in I}$ is compatible. However, it does not arise from a point of $\{0,12,3\}$: since the $k_i$ must separate points, there must be $i$ with $k_i(0) = k_i(12) \neq k_i(3)$, and another $i$ with $k_i(0)\neq k_i(12) = k_i(3)$. Hence neither of $x\in\{0,12,3\}$ results in the given compatible family, and the cone $\{k_i\}$ is not effective-monic.
\end{example}

\begin{figure}
\[
\xymatrix@=.2cm{	&	&		\bullet	\ar@{}[l]|3		&		&	&	&	&	\bullet	\ar@{}[l]|3	\\
			&	&		\bullet	\ar@{}[l]|2		& \ar[rrr]^f	&	&	&	&	\bullet	\ar@{}[l]|2	\\
			\bullet	\ar@{}[u]|0	& \bullet \ar@{}[u]|1	&	&		&	&	&	&	\bullet	\ar@{}[l]|{01}	\\
																			\\
			&	\ar[ddd]_g														\\	
																			\\
			&																\\
			&			&			&										\\
			\bullet	\ar@{}[u]|0	& \bullet \ar@{}[u]|1	& \bullet \ar@{}[u]|{23}							}
\]
\caption{Illustration of the cone $\{f,g\}$ of Example~\ref{ex4to3}.}
\label{fig4to3}
\end{figure}
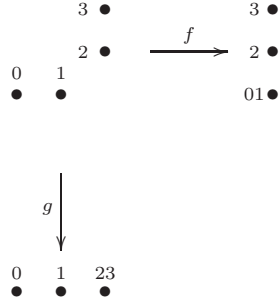

Incidentally, the cone $\{f',g'\}$ from above is arguably the simplest example of a cone that separates points (is jointly injective) without being effective-monic.

\begin{remark}
The previous example can also be understood in terms of effectus theory~\cite[Assumption~1]{jacobs}: the relevant pushout square is of the form
\[
\xymatrix{ W + Y \ar[r]^{\id + f} \ar[d]_{g + \id} & W + Z \ar[d]^{g + \id} \\
	  X + Y \ar[r]_{\id + f} & X + Z }
\]
where `$+$' is the coproduct in $\CHs$ and both $f$ and $g$ are the unique map $\mathbf{2}\to\mathbf{1}$. In general, any cone consisting of $\id+f:W+Y\to W+Z$ and $g+\id:W+Y\to X+Y$ is effective-monic by Lemma~\ref{malcev}.

It is conceivable that there are deeper connections with effectus theory than just at the level of examples, but so far we have not explored this theme any further.
\end{remark}

Starting to get back to C*-algebras, we record one further statement about cones for further use.

\begin{lemma}
A cone $\{f_i : X\to Y_i\}$ separates points if and only if the ranges of the $C(f_i) : C(Y_i)\to C(X)$ generate $C(X)$ as a C*-algebra.
\label{seppoints}
\end{lemma}

\begin{proof}
By the Stone-Weierstrass theorem, the C*-subalgebra generated by the ranges of the $C(f_i)$ equals $C(X)$ if and only if it separates points (as a subalgebra). This C*-subalgebra is generated by the elements $g_i\circ f_i\in C(X)$, where $g_i : Y_i\to [0,1]$ ranges over all functions, and hence the subalgebra separates points if and only if these functions separate points. This in turn is equivalent to the $f_i$ separating points, since the $g_i : Y_i\to [0,1]$ also separate points.
\end{proof}

\subsection*{How to guarantee commutativity?}
\label{guarcommsec}

The previous subsection was concerned with sheaf conditions satisfied by the functors $-(A)$ for commutative $A$. Now, we want to investigate which of these sheaf conditions hold for general $A$.

\begin{definition}
\label{guarcommdef}
An effective-monic cone $\{f_i:X\to Y_i\}_{i\in I}$ in $\CHs$ is \emph{guaranteed commutative} if every functor $-(A)$ satisfies the sheaf condition on it.
\end{definition}

In detail, $-(A)$ satisfies the sheaf condition on $\{f_i\}$ if and only if restricting a $*$-homomorphism $\alpha:C(X)\to A$ along all $C(f_i):C(Y_i)\to C(X)$ to families $\beta_i : C(Y_i)\to A$ that are compatible in the sense that $\beta_i\circ C(g) = \beta_j\circ C(h)$ for every diagram of the form~\eqref{compdiag},
\[
\xymatrix{	X \ar[r]^{f_i} \ar[d]_{f_j}	& Y_i \ar[d]^g	\\
		Y_j \ar[r]_h			& Z		}
\]
results in a bijection. In terms of the functor $C:\CHs^\op\to\Calg$, this holds if and only if the diagram
\[
\xymatrix{	\vdots									& C(Y_i) \ar[drr]		\\
		C(Y_i)\pullback{C(f_i)}{C(f_j)}C(Y_j) \ar[ur]^{C(f_i)} \ar[dr]_{C(f_j)}	& \vdots	&& C(X)		\\
		\vdots									& C(Y_j) \ar[urr]		}
\]
which is the image of~\eqref{pointssheaf} under $C$, is a colimit in $\Calg$. Here, we have used the canonical isomorphism $C(Y_i\pushout{f_i}{f_j} Y_j)\cong C(Y_i)\pullback{C(f_i)}{C(f_j)} C(Y_j)$, which holds because $C$ is a right adjoint. So we are dealing with an instance of the question, which limits does $C$ turn into colimits?

\begin{remark}
\label{gluemeas}
In terms of the physical interpretation of Remarks~\ref{measurements} and~\ref{reconstruct}, the sheaf condition on a cone $\{f_i:X\to Y_i\}$ states that every compatible family of measurements with outcomes in the $Y_i$ corresponds to a unique measurement with values in $X$ which coarse-grains to the given measurements via the $f_i$. 
\end{remark}

The terminology of Definition~\ref{guarcommdef} is motivated by the following observation:

\begin{lemma}
An effective-monic cone $\{f_i:X\to Y_i\}_{i\in I}$ is guaranteed commutative if and only if for every $A\in\Calg$ and compatible family $\beta_i : C(Y_i)\to A$, the ranges of the $\beta_i$ commute.
\label{gclem}
\end{lemma}

\begin{proof}
Suppose that the criterion holds. For $A\in\Calg$, we show that restricting a $*$-homomorphism $C(X)\to A$ to a compatible family of $*$-homomorphisms $C(Y_i)\to A$ is a bijection. We first show injectivity, so let $\alpha,\alpha' :  C(X)\to A$ be such that the resulting families coincide, $\beta_i = \beta'_i$. In particular, this means that the range of each $\beta_i$ coincides with the range of $\beta'_i$, and hence $\im(\alpha) = \im(\alpha')$ by Lemma~\ref{seppoints}. Hence we are back in the commutative case, where Gelfand duality and the effective-monic assumption apply.

For surjectivity, let a compatible family $\beta_i : C(Y_i)\to A$ be given. By assumption, there is some commutative subalgebra $B\subseteq A$ which contains the ranges of all $\beta_i$, and it is sufficient to prove the sheaf condition with $B$ in place of $A$. The claim then follows from Gelfand duality together with the assumption that $\{f_i\}$ is effective-monic.

Conversely, if the sheaf condition holds on a functor $-(A)$, then the $\beta_i : C(Y_i)\to A$ all arise from restricting some $\alpha : C(X)\to A$ along $C(f_i):C(Y_i)\to C(X)$. In particular, the range of every $\beta_i$ is contained in the range of $\alpha$, which is a commutative C*-subalgebra.
\end{proof}

The crucial ingredient here is the fact that commutativity is a pairwise property, in the sense that if any family of elements in a C*-algebra commute pairwise, then they generate a commutative C*-subalgebra. We will meet this property again in Definition~\ref{piecewisedef}.

In the sense of Lemma~\ref{gclem}, the question is under what conditions an effective-monic cone `guarantees commutativity' of the ranges of a compatible family.

\begin{example}
\label{exgc}
The effective-monic cone of Example~\ref{ex4to3} is guaranteed commutative: in terms of indicator functions of individual points, the compatibility assumption on a pair of $*$-homomorphisms $\beta_f:C(\{01,2,3\})\to A$ and $\beta_g:C(\{0,1,23\}) \to A$ is that
\[
\beta_f(\chi_{01}) = \beta_g(\chi_0) + \beta_g(\chi_1),\qquad \beta_g(\chi_{23}) = \beta_f(\chi_2) + \beta_f(\chi_3).
\]
So $\beta_g(\chi_0)$ is a projection below $\beta_f(\chi_{01})$, and in particular orthogonal to $\beta_f(\chi_2)$ and $\beta_f(\chi_3)$, so that it commutes with every element in the range of $\beta_f$. Proceeding like this proves that the ranges of $\beta_f$ and $\beta_g$ commute entirely.
\end{example}

\begin{example}
Let $\Tl\subseteq\Cl$ be the unit circle, and $p_\Re,p_\Im : \Tl\to [-1,+1]$ the two coordinate projections. Then the cone $\{p_\Re,p_\Im\}$ is effective-monic by Lemma~\ref{malcev}, or alternatively since applying $p_\Re$ and $p_\Im$ establishes a bijection between points of $x$ and pairs of numbers $y_\Re,y_\Im\in[-1,+1]$ with $y_\Re^2 + y_\Im^2 = 1$. Hence compatible families $\{\beta_\Re,\beta_\Im\}$ are $*$-homomorphisms $\beta_\Re:C([-1,+1])\to A$ and $\beta_\Im:C([-1,+1])\to A$ that correspond to self-adjoint elements $\beta_\Re(\id),\beta_\Im(\id)\in [-1,+1](A)$ with $\beta_\Re(\id)^2 + \beta_\Im(\id)^2 = 1$. Such a pair of self-adjoints arises from a unitary by functional calculus if and only if they commute. For example, choosing any $A$ with non-commuting symmetries $s_\Re$ and $s_\Im$ provides a compatible family that does not arise in this way upon putting $\beta_\Re\defin s_\Re/\sqrt{2}$ and $\beta_\Im\defin s_\Im/\sqrt{2}$. Therefore $\{p_\Re,p_\Im\}$ is not guaranteed commutative.
\label{circleproj}
\end{example}

So far, we know of one powerful sufficient condition for guaranteeing commutativity:

\begin{definition}
An effective-monic cone $\{f_i:X\to Y_i\}_{i\in I}$ in $\CHs$ is \emph{directed} if for every $i\in I$ there is a cone $\{g_i^j : Y_i\to Z_i^j\}_{j\in J_i}$ which separates points, and such that for every $i,i'\in I$ and $j\in J_i$, $j'\in J_{i'}$ there is $k\in I$ and a diagram
\be
\begin{split}
\label{directeddiag}
\xymatrix{				& X \ar[dl]_{f_i} \ar[d]|{f_k} \ar[dr]^{f_{i'}}					\\
		Y_i \ar[d]_{g_i^j}	& Y_k \ar[dl] \ar[dr]				& Y_{i'} \ar[d]^{g_{i'}^{j'}}	\\
		Z_i^j			&						& Z_{i'}^{j'}			}
\end{split}
\ee
\label{directeddef}
\end{definition}

Note that this definition can be considered in principle in any category.

\begin{proposition}
If $\{f_i\}$ is effective-monic and directed, then it is also guaranteed commutative.
\label{guarcommcrit}
\end{proposition}

\begin{proof}
By Lemma~\ref{gclem}, it is enough to show that the ranges of a compatible family $\{\beta_i:C(Y_i)\to A\}$ commute. By Lemma~\ref{seppoints}, it is enough to prove that the range of $\beta_i\circ C(g_i^j) : C(Z_i^j)\to A$ commutes with the range of $\beta_{i'}\circ C(g_{i'}^{j'}): C(Z_{i'}^{j'})\to A$ for any $i,i'\in I$ and $j\in J_i$, $j'\in J_{i'}$. Thanks to~\eqref{directeddiag} and the compatibility, both of these ranges are contained in the range of $\beta_k : C(Y_k)\to A$, which is commutative.
\end{proof}

\newcommand{\cantorspace}{\mathbf{2}^\Nl}

\begin{example}
\label{cofiltered}
Let $\cantorspace$ be the Cantor space, with projections $p_n:\cantorspace\to\mathbf{2}^n$ for every $n\in\Nl$. Then the cone $\{p_n\}_{n\in\Nl}$ is effective-monic and directed. Therefore it is also guaranteed commutative.

More generally, let $\Lambda$ be a small cofiltered category and $L:\Lambda\to\CHs$ a functor of which we consider the limit $\lim_\Lambda L\in\CHs$. The cone of limit projections $\{p_\lambda : \lim_\Lambda L \to L(\lambda)\}$ is effective-monic (Example~\ref{limit}). With the trivial cones $\{\id\}$ on the codomains $L(\lambda)$, the cofilteredness implies that the cone is also directed, and therefore guaranteed commutative. What we have shown hereby in a roundabout manner is that a filtered colimit of commutative C*-algebras is again commutative.
\end{example}

Unfortunately, the converse to Proposition~\ref{guarcommcrit} is not true:

\begin{example}
The effective-monic cone $\{f,g\}$ of Examples~\ref{ex4to3} and~\ref{exgc} is not directed, despite being guaranteed commutative. The reason is that the additional cones as in Definition~\ref{directeddef} would have to contain some $h:\{12,3,4\}\to Z_{12}$ with $h(3)\neq h(4)$, and similarly some $k:\{1,2,34\}\to Z_{34}$ with $k(1)\neq k(2)$. By~\eqref{directeddiag}, this would mean that the cone $\{f,g\}$ would have to contain a function that separates both $1$ from $2$ and $3$ from $4$, which is not the case.
\end{example}

So while Proposition~\ref{guarcommcrit} is sufficiently powerful for the remainder of this paper, it remains open to find a necessary and sufficient condition for guaranteeing commutativity.

\begin{lemma}
For any $X\in\CHs$, the cone $\{f:X\to\Box\}$ of all functions $f:X\to\Box$ is directed.
\label{guarcommsquare}
\end{lemma}

By Lemma~\ref{squarecover}, we already know that this cone is effective-monic. By Proposition~\ref{guarcommcrit}, we can now conclude that it also is guaranteed commutative.

\begin{proof}
In Definition~\ref{directeddef}, take every $\{g_i^j\}_{j\in J_i}$ to be the cone consisting of all functions $\Box\to[0,1]$. Since the pairing of any two functions $X\to[0,1]$ is a function $X\to\Box$, the cone $\{f:X\to\Box\}$ is directed. 
\end{proof}

\begin{remark}
In terms of Remark~\ref{gluemeas}, this lemma `explains' why physical measurements are numerical: for every conceivable measurement with values in some arbitrary space $X$, conducting that measurement and recording the outcome in $X$ is equivalent to conducting a sufficient number of measurements with values in $\Box$ and recording their outcomes, which are now plain (complex) numbers.
\end{remark}

\begin{lemma}
If two cones $\{f_i:W\to Y_i\}_{i\in I}$ and $\{g_j:X\to Z_i\}_{j\in J}$ are effective-monic and directed, then so is the product cone
\[
\{ f_i\times g_j : W\times X\to Y_i\times Z_j \}_{(i,j)\in I\times J}.
\]
\label{productcovers}
\end{lemma}

\begin{proof}
Let $\{h^k_i:Y_i\to U_i^k\}_{k\in K_i}$ and $\{k^l_j:Z_j\to V_j^l\}_{l\in L_j}$ be the families of additional cones that witness the directedness. Then for $(i,j)\in I\times J$, consider the cone at $Y_i\times Z_j$ given by
\be
\label{productsep}
\{ h^k_i p_{Y_i} : Y_i\times Z_j\to U_i^k \} \cup \{ k^l_j p_{Z_j} : Y_i\times Z_j\to V_j^l \}
\ee
with index set $K_i\amalg L_j$. This cone separates the points of $X_i\times Y_j$, since any two different points differ in at least one coordinate. The condition of Definition~\ref{directeddef} is easy to check by distinguishing the cases of the left and the right morphism in~\eqref{directeddiag} belonging to either part of~\eqref{productsep}. The only interesting case that comes up is when one considers a $h^k_i p_{Y_i} : Y_i\times Z_{j'}\to U_i^k$ together with a $k^l_j p_{Z_j} : Y_{i'}\times Z_j\to \times V_j^l$, resulting in a diagram of the form
\[
\xymatrix{						& W\times X \ar[dl]_{f_i\times g_{j'}} \ar[d]|{} \ar[dr]^{f_{i'}\times g_j}							\\
		Y_i\times Z_{j'} \ar[d]_{h_i^k p_{Y_i}}	& \ar[dl] \ar[dr]							& Y_{i'}\times Z_j	\ar[d]^{k_j^l p_{Z_j}}	\\
		U_i^k		&											& V_j^l				}
\]
where indeed the central vertical arrow can be taken to be $f_i\times g_j$.
\end{proof}

In combination with Lemma~\ref{guarcommsquare}, we therefore obtain:

\begin{corollary}
For any $X,Y\in\CHs$, the cone $\{f\times g:X\times Y\to\Box\times\Box\}$ indexed by all functions $f:X\to\Box$ and $g:Y\to\Box$ is directed.
\label{doublesquare}
\end{corollary}

Another simple class of examples is as follows:

\begin{lemma}
Let $\{f_i : X\to Y_i\}_{i\in I}$ be an effective-monic cone on $X\in\CHs$. Then the cone
\[
\left\{ (f_{i_1},\ldots,f_{i_n}) : X\to Y_{i_1}\times\ldots\times Y_{i_n} \right\}
\]
consisting of all finite tuplings of the $f_i$ is effective-monic and directed.
\label{Sinfty}
\end{lemma}

Alternatively, we could phrase this as saying that if an effective-monic cone is closed under pairing, then it is directed.

\begin{proof}
Mapping points of $X$ to compatible families of points in all finite products $\prod_{m=1}^n Y_{i_m}$ is trivially injective, since it already is so on single-factor products due to the effective-monic assumption. Concerning surjectivity, the compatibility assumption guarantees that the component $(y_1,\ldots,y_n)\in Y_{i_1}\times\ldots\times Y_{i_n}$ is uniquely determined by the components in every individual $y_i$, since this is precisely the compatibility condition on diagrams of the form
\[
\xymatrix{	X \ar[d]_{f_{i_m}} \ar[rr]^(.4){(f_{i_1},\ldots,f_{i_n})}	&& \mathlarger{\mathlarger{\prod}}_{m=1}^n Y_{i_m} \ar[d]^{p_m}	\\
		Y_{i_m} \ar@{=}[rr]					&& Y_{i_m}				}
\]
Hence the new cone is also effective-monic.

The condition of Definition~\ref{directeddef} holds by construction, with the trivial cone $\{\id\}$ on the codomains.
\end{proof}

Next, we briefly investigate the collection of directed effective-monic cones in its entirety.

\begin{proposition}
\label{nocoverage}
The collection of all directed effective-monic cones on $\CHs$ is not a coverage.
\end{proposition}

Results along the lines of~\cite[Theorem~1.1]{reyes} indicate that this is not due to the potential inadequacy of our definitions, but rather due to fundamental obstructions related to the noncommutativity.

\begin{proof}
Consider $X\defin\{0,1\}^3$ with the three product projections $p_1,p_2,p_3:\{0,1\}^3\to\{0,1\}$. By reasoning analogous to the proof of Lemma~\ref{squarecover}, their three pairings
\be
\label{facecone}
\left\{\:(p_1,p_2),\: (p_1,p_3),\: (p_2,p_3) \: : \: \{0,1\}^3\longrightarrow\{0,1\}^2 \:\right\}
\ee
form an effective-monic cone. By reasoning analogous to the proof of Lemma~\ref{guarcommsquare}, this cone is directed.

Now consider the function $f:\{0,1\}^3\to\mathbf{4}$ defined by mapping every element of $\{0,1\}^3$ to the sum of its digits. In any square of the form
\[
\xymatrix{	\{0,1\}^3 \ar[r]^{(p_1,p_2)} \ar[d]_f	& \{0,1\}^2 \ar[d]^g	\\
		\mathbf{4} \ar[r]_h			& Z			}
\]
we necessarily have
\[
\underbrace{h(f(000))}_{=h(0)} = g(00) = h(f(001)) = h(f(010)) = g(01) = \underbrace{h(f(010))}_{=h(1)} = \ldots = \underbrace{h(f(110))}_{=h(2)} = \ldots = \underbrace{h(f(111))}_{=h(3)},
\]
and therefore $h$ must be constant. By symmetry, the same must hold with $(p_1,p_3)$ or $(p_2,p_3)$ in place of $(p_1,p_2)$. Hence any cone on $\mathbf{4}$ that factors through~\eqref{facecone} must identify \emph{all} points of $\mathbf{4}$. In particular, no such cone can even be effective-monic, let alone directed.
\end{proof}

We close this subsection with another potential criterion for guaranteeing commutativity. This is not relevant for the remainder of the paper.

\begin{lemma}
The following conditions on a cone $\{f_i:X\to Y_i\}$ in $\CHs$ are equivalent:
\begin{enumerate}
\item\label{lfpoint} For every $x\in X$ and neighbourhood $U\ni x$ there exists $i\in I$ with
\[
f_i^{-1}(f_i(x)) \subseteq U.
\]
\item\label{lfngbhd} For every $x\in X$ and neighbourhood $U\ni x$ there exist $i\in I$ and a neighbourhood $V\ni f_i(x)$ with
\[
f_i^{-1}(V) \subseteq U.
\]
\item\label{lfbasis} The sets of the form $f_i^{-1}(V)$ for open $V\subseteq Y_i$ form a basis for the topology on $X$.
\end{enumerate}
\end{lemma}

\begin{proof}
\begin{enumerate}
\item[\implproof{lfpoint}{lfngbhd}] Since $X\setminus U$ is compact, $f_i(X\setminus U)$ is a closed set, and disjoint from $\{x\}$ by assumption. Now take $V$ to be any open neighbourhood of $f_i(x)$ disjoint from $f_i(X\setminus U)$.
\item[\implproof{lfngbhd}{lfbasis}] Suppose $x\in f_i^{-1}(V_i)\cap f_j^{-1}(V_j)$. Then by assumption, there is $k$ and an open $V_k\subseteq Y_k$ with $f_k(x)\in V_k$ such that
\[
f_k^{-1}(V_k) \subseteq f_i^{-1}(V_i)\cap f_j^{-1}(V_j).
\]
\item[\implproof{lfbasis}{lfpoint}] There must be a basic open $f_i^{-1}(V_i)$ with $x\in f_i^{-1}(V_i)\subseteq U$.\qedhere
\end{enumerate}
\end{proof}

\begin{definition}
If the above conditions hold, we say that the cone $\{f_i\}$ is \emph{locally injective}.
\end{definition}

Clearly, a locally injective cone separates points. However, it is not necessarily effective-monic:

\begin{example}
The cone consisting of all three surjective functions $\mathbf{3}\to \mathbf{2}$ is locally injective. However, it is not effective-monic: the pushout of any two different maps $\mathbf{3}\to\mathbf{2}$ is trivial, and hence there are $2^3$ compatible families of points in the cone, but only $3$ points in $X$.
\end{example}

\begin{example}
The cone $\{p_\Re,p_\Im\}$ from Example~\ref{circleproj} is not locally injective: for any angle $0<\varphi<\pi/2$, the point $(\cos\varphi,\sin\varphi)\in \Tl$ cannot be distinguished from $(\cos\varphi,-\sin\varphi)\in \Tl$ under $p_\Re$, and not from $(-\cos\varphi,\sin\varphi)$ under $p_\Im$. 
\end{example}

\begin{conjecture}
\label{linjconj}
An effective monic cone $\{f_i:X\to Y_i\}$ that is locally injective is also guaranteed commutative.
\end{conjecture}

Since the cone of all functions $X\to\Box$ is an effective-monic and locally injective cone, proving this conjecture would again show that $\{f:X\to\Box\}$ is guaranteed commutative. Furthermore, this would detect some cones as guaranteed commutative that are not detected as such by Proposition~\ref{guarcommcrit}: the effective-monic cone of Examples~\ref{ex4to3} and~\ref{exgc} is one of these.

\begin{example}
In the setting of Example~\ref{cofiltered}, the topology of $\lim_\Lambda L$ is generated by the preimages of opens in all the $L(\lambda)$. The cofilteredness assumption implies that these opens form a basis: for $U_\lambda \subseteq L(\lambda)$ and $U_{\lambda'}\subseteq L(\lambda')$, we have $\hat{\lambda}$ and morphisms $f:\hat{\lambda}\to\lambda$ and $f':\hat{\lambda}\to\lambda'$ such that
\[
\xymatrix{ 			& \lim_\Lambda L \ar[dl]_{p_\lambda} \ar[d]|{p_{\hat{\lambda}}} \ar[dr]^{p_{\lambda'}} 				\\
		L(\lambda)	& L(\hat{\lambda}) \ar[l]^{L(f)} \ar[r]_{L(f')}							& L(\lambda')	}
\]
commutes. In particular, $f^{-1}(U_\lambda)\cap f'^{-1}(U_{\lambda'})$ is an open in $L(\hat{\lambda})$ whose preimage in $\lim_\Lambda L$ is exactly the intersection of the preimages of $U_\lambda$ and $U_{\lambda'}$. Hence the limit cone $\{p_\lambda\}$ is also locally injective. By Example~\ref{cofiltered}, this is in accordance with Conjecture~\ref{linjconj}.
\end{example}

Similar to the situation with Proposition~\ref{guarcommcrit}, being locally injective is also not a necessary condition for guaranteeing commutativity:

\newcommand{\cube}{\text{\mancube}}

\begin{example}
There are effective-monic cones that are directed and hence guaranteed commutative, but not locally injective. For example with $\cube\defin[0,1]^3$ the unit cube, the three face projections $p_1,p_2,p_3 : \cube\to\Box$ form a cone $\{p_1,p_2,p_3\}$ that is effective-monic but not locally injective. Nevertheless, considering copies of the cone $\{p_\Re,p_\Im:\Box\to [0,1]\}$ in Definition~\ref{directeddef} shows that the cone is directed, and hence guaranteed commutative. In particular, the converse to Conjecture~\ref{linjconj} is false.
\end{example}

\subsection*{The category of sheaves and its smallness properties}

Now that we have some idea of which sheaf conditions are satisfied by C*-algebras, we investigate completely general functors $\CHs\to\Sets$ satisfying (some of) these sheaf conditions.

\begin{definition}
A functor $F:\CHs\to\Sets$ is a \emph{sheaf} if it satisfies the sheaf condition on all effective-monic cones that are directed.
\label{sheafdef}
\end{definition}

We write $\Sh(\CHs)$ for the resulting category of sheaves, which is a full subcategory of $\Sets^\CHs$. Due to Proposition~\ref{nocoverage}, the sheaf conditions are \emph{not} those of a (large) site. Nevertheless, we expect that $\Sh(\CHs)$ is an instance of a category of sheaves on a \emph{quasi-pretopology} or on a \emph{Q-category}, whose categories of sheaves were investigated by Kontsevich and Rosenberg in the context of noncommutative algebraic geometry~\cite{KR1,KR2}\footnote{It is natural to suspect that the reason for why Grothendieck topologies do not apply is in both cases due to the noncommutativity, as has been formally proven in~\cite{reyes}. However, so far we have not explored the relation to the work of Kontsevich and Rosenberg any further.}.

A priori, $\Sh(\CHs)$ may seem rather unwieldy, and it is not even clear whether it is locally small.

\begin{lemma}
\label{evalbox}
Let $F,G\in\Sh(\CHs)$. Evaluating natural transformations on $\Box$ is injective,
\[
\xymatrix{	\Sh(\CHs)(F,G) \: \ar@{^{(}->}[r]	& \: \Sets(F(\Box),G(\Box)).	}
\]
\end{lemma}

\begin{proof}
Since $F$ and $G$ satisfy the sheaf condition on $\{f:X\to\Box\}$ by Corollary~\ref{guarcommsquare}, the canonical map
\[
\xymatrix{	F(X) \ar[r]	& \mathlarger{\mathlarger{\prod}}_{f:X\to\Box} F(\Box)	}
\]
is injective. Hence for any $\eta:F\to G$, the naturality diagram
\[
\xymatrix{	F(X) \ar[r] \ar[d]_{\eta_X}	& \mathlarger{\mathlarger{\prod}}_{f:X\to\Box} F(\Box) \ar[d]^{\prod_f \eta_\Box}	\\
		G(X) \ar[r]			& \mathlarger{\mathlarger{\prod}}_{f:X\to\Box} G(\Box)				}
\]
shows that every component $\eta_X$ is uniquely determined by $\eta_\Box$.
\end{proof}

\begin{corollary}
$\Sh(\CHs)$ is locally small.
\end{corollary}

\begin{proof}
Lemma~\ref{evalbox} provides an upper bound on the size of each hom-set. 
\end{proof}

With functors $-(A)$ and $-(B)$ for $A,B\in\Calg$ in place of $F$ and $G$, Lemma~\ref{evalbox} also follows from the Yoneda lemma and the fact that $C(\Box)$ is a separator in $\Calg$. The latter is true more generally:

\begin{corollary}
$-(C(\Box))$ is a separator in $\Sh(\CHs)$.
\label{locsmall}
\end{corollary}

Recall that as functors $\CHs\to\Sets$, we have $-(C(\Box)) \cong \CHs(\Box,-)$.

\begin{proof}
By the Yoneda lemma,
\be
\label{yonedaeq}
\Sh(\CHs)(-(C(\Box)),F) = \Sets^\CHs(\CHs(\Box,-),F) = F(\Box),
\ee
and hence the claim follows from Lemma~\ref{evalbox}.
\end{proof}

The following stronger injectivity property will play a role in the next section:

\begin{lemma}
For $F\in\Sh(\CHs)$, the following are equivalent:
\begin{enumerate}
\item\label{boxinj} The canonical map
\be
\label{doublebox}
\xymatrix{	(F(p_1),F(p_2)) \: :\: F(\Box\times\Box) \ar[r]	& F(\Box)\times F(\Box)	}
\ee
is injective.
\item\label{allinj} For every $X\in\CHs$ and effective-monic $\{f_i:X\to Y_i\}$, the canonical map
\[
\xymatrix{	F(X) \ar[r]	& \mathlarger{\mathlarger{\prod}}_{i\in I} F(Y_i)	}
\]
is injective.
\end{enumerate}
\label{injlem}
\end{lemma}

In~\ref{allinj}, the point is that the cone may not be directed, so generically $F$ does not satisfy the sheaf condition on it. The intuition behind the lemma is that these (equivalent) conditions hold in the C*-algebra case, and then the image of~\eqref{doublebox} consists of precisely the pairs of commuting normal elements. In terms of the interpretation as measurements on a physical system, this image consists of the pairs of measurements (with values in $\Box$) that are jointly measurable.

In the proof, we can start to put the seemingly haphazard lemmas of the previous subsection to some use.

\begin{proof}
Since the cone $\{p_1,p_2:\Box\times\Box\to\Box\}$ is effective-monic, condition~\ref{boxinj} is a special case of~\ref{allinj}.

In the other direction, we first show that for every $X,Y\in\CHs$, the canonical map $F(X\times Y)\to F(X)\times F(Y)$ is injective. By Corollary~\ref{doublesquare}, the left vertical arrow in
\[
\xymatrix{	F(X\times Y) \ar[r] \ar[d]			& F(X) \times F(Y) \ar[d]	\\
		\mathlarger{\mathlarger{\prod}}_{f:X\to\Box,g:Y\to\Box} F(\Box\times\Box) \ar[r]	& \left(\mathlarger{\mathlarger{\prod}}_{f:X\to\Box} F(\Box)\right)\times\left(\mathlarger{\mathlarger{\prod}}_{g:Y\to\Box} F(\Box)\right)		}
\]
is injective. Since the lower horizontal arrow is injective by assumption, it follows that the upper horizontal arrow is also injective. By induction, we then obtain that $F(\prod_{j=1}^n X_j)\to \prod_{j=1}^n F(X_j)$ is injective for any finite product.

Now let $\{f_i\}$ be an arbitrary effective-monic cone on $X$. By Lemma~\ref{Sinfty}, $F$ satisfies the sheaf condition on the cone consisting of all the finite tuplings $(f_{i_1},\ldots,f_{i_n})$. Hence we have the diagram
\[
\xymatrix{	F(X) \ar[r] \ar[d]	& \mathlarger{\mathlarger{\prod}}_{i\in I} F(Y_i) \ar[d]								\\
		\mathlarger{\mathlarger{\prod}}_{n\in\Nl} \: \mathlarger{\mathlarger{\prod}}_{i_1,\ldots,i_n\in I} F\left(\mathlarger{\mathlarger{\prod}}_{m=1}^n Y_{i_m}\right) \ar[r]	& \mathlarger{\mathlarger{\prod}}_{n\in\Nl} \: \mathlarger{\mathlarger{\prod}}_{i_1,\ldots,i_n\in I} \: \mathlarger{\mathlarger{\prod}}_{m=1}^n F(Y_{i_m})	}
\]
where the left vertical arrow is injective due to the sheaf condition, and the lower horizontal one due to the first part of the proof. Hence also the upper horizontal arrow is injective.
\end{proof}

So far, we do not know of any sheaf $\CHs\to\Sets$ that would \emph{not} have the property characterized by the lemma.

By Gelfand duality, the commutative C*-algebras are precisely the representable functors $\CHs(W,-):\CHs\to\Sets$. These are characterized in terms of a condition similar to the previous lemma:

\begin{lemma}
\label{replem}
For $F\in\Sh(\CHs)$, the following are equivalent:
\begin{enumerate}
\item\label{boxbij} The canonical map
\[
\xymatrix{	(F(p_1),F(p_2)) \: :\: F(\Box\times\Box) \ar[r]	& F(\Box)\times F(\Box)	}
\]
is bijective.
\item\label{allbij} $F$ satisfies the sheaf condition on every effective-monic cone $\{f_i:X\to Y_i\}$ in $\CHs$.
\item\label{Frep} $F$ is representable.
\end{enumerate}
\label{bijlem}
\end{lemma}

\begin{proof}
By the definition of effective-monic,~\ref{Frep} trivially implies~\ref{allbij}. Also if~\ref{allbij} holds, then it is easy to show~\ref{boxbij}: the empty cone is effective-monic on $1\in\CHs$, which implies $F(1)\cong 1$. With this in mind,~\ref{boxbij} is the sheaf condition on the effective-monic cone $\{p_1,p_2:\Box\times\Box\to\Box\}$.

The burden of the proof is the implication from~\ref{boxbij} to~\ref{Frep}. By the representable functor theorem~\cite[p.~130]{ML} and the generation of limits by products and equalizers, it is enough to show that $F$ preserves products and equalizers, which we do in several steps. First, the functor $-\times Y$ preserves pushouts for any $Y\in\CHs$: the functor $-\times Y:\CGHs\to\CGHs$ preserves colimits as a left adjoint~\cite[Theorem~VII.8.3]{ML}, and the inclusion functor $\CHs\to\CGHs$ also preserves finite colimits, since it preserves finite coproducts and coequalizers (the latter by the automatic compactness of quotients of compact spaces).

Second, we prove that the canonical map $F(X\times\Box)\longrightarrow F(X)\times F(\Box)$ is a bijection for every $X\in\CHs$. To this end, we consider the effective-monic cone $\{f\times\id_\Box : X\times\Box\to\Box\times\Box\}$ indexed by $f:X\to\Box$. We know that this cone is directed by Lemmas~\ref{guarcommsquare} and~\ref{productcovers}. This entails that $F(X\times\Box)$ is equal to the set of compatible families $\{\beta_f\}_{f:X\to\Box}$ of elements of $\prod_{f:X\to\Box} F(\Box\times\Box)$. Since $-\times\Box$ preserves pushouts as per the first observation, the compatibility condition is the one associated to the squares of the form
\[
\xymatrix{	X\times\Box \ar[r]^{f\times\id_\Box} \ar[d]_{g\times\id_\Box}	& \Box\times\Box \ar[d]					\\
		\Box\times\Box \ar[r]						& (\Box\pushout{f\times\id}{g\times\id}\Box)\times\Box	}
\]
By using the fact that the maps $\Box\pushout{f\times\id}{g\times\id}\Box\longrightarrow\Box$ separate points, it is sufficient to postulate the compatibility on all commuting squares of the form
\[
\xymatrix{	X\times\Box \ar[r]^{f\times\id_\Box} \ar[d]_{g\times\id_\Box}	& \Box\times\Box \ar[d]^{h\times\id_\Box}	\\
		\Box\times\Box \ar[r]_{k\times\id_\Box}				& \Box\times\Box				}
\]
So upon decomposing $\beta_f = (\beta_f^1,\beta_f^2)$ via $F(\Box\times\Box) = F(\Box)\times F(\Box)$, the compatibility condition is precisely that $F(h)(\beta^1_f) = F(k)(\beta^1_g)$ and that $\beta^2_f = \beta^2_g$ for all $h,k:\Box\to\Box$ with $hf=kg$. Since the family of first components therefore corresponds precisely to an element of $F(X)$, we conclude that the canonical map $F(X\times\Box)\longrightarrow F(X)\times F(\Box)$ is an isomorphism.

Third, we use this result to show that $F(X\times Y)\longrightarrow F(X)\times F(Y)$ is an isomorphism for all $X,Y\in\CHs$; the proof is the same as above, just with $-\times\Box$ replaced by $-\times Y$. The case of finite products $F(\prod_{i=1}^n X_i)\cong \prod_{i=1}^n F(X_i)$ then follows by induction, and the case of infinite products by the sheaf condition.

The preservation of equalizers also takes a bit of work. Since every monomorphism $f:X\to Y$ in $\CHs$ is regular, the singleton cone $\{f\}$ is effective-monic. Since this cone is trivially directed, $F$ satisfies the sheaf condition on it, which entails that $F(f):F(X)\to F(Y)$ must be injective.

Second, a diagram
\[
\xymatrix{	E \ar[r]|e	& X \ar@<1ex>[r]^f \ar@<-1ex>[r]_g	& Y	}
\]
is an equalizer if and only if
\[
\xymatrix{	E \ar[r]^e \ar[d]_e	& X \ar[d]^{(\id_X,f)}	\\
		X \ar[r]_(.4){(\id_X,g)}	& X\times Y		}
\]
is a pullback. By constructing the pushout $X\pushout{e}{e}X$ as a quotient of $X\amalg X$ and doing a case analysis on pairs of points in $X\pushout{e}{e}X$, the induced arrow $k$ in
\[
\xymatrix{	E \ar[r]^e \ar[d]_e				& X \ar[d]^(.45)i \ar@/^3ex/[ddr]^{(\id_X,f)}			\\
		X \ar[r]_(.4)j \ar@/_3ex/[rrd]_{(\id_X,g)}	& X\pushout{e}{e}X \ar@{-->}[dr]|k				\\
								&						& X\times Y	}
\]
is seen to be a monomorphism, and therefore so is $F(k)$. So if $\beta\in F(X)$ is such that $F(f)(\beta) = F(g)(\beta)$, then also $F(i)(\beta) = F(j)(\beta)$. But by the sheaf condition on the singleton cone $\{e\}$, this means that $\beta$ is in the image of $F(e)$, as was to be shown.
\end{proof}

For $F\in\Sh(\CHs)$, any $W\in\CHs$ and any $\alpha\in F(W)$, let $F_\alpha : \CHs\to\Sets$ be the subfunctor of $F$ generated by $\alpha$. Concretely, over every $X\in\CHs$, the set $F_\alpha(X)$ consists of all the images $F(f)(\alpha)$ for $f:W\to X$.

\begin{proposition}
If the canonical map
\be
\label{timesinj}
\xymatrix{	F(\Box\times\Box) \ar[r]	& F(\Box)\times F(\Box)	}
\ee
is injective, then such an $F_\alpha$ is representable.
\label{singlygenerated}
\end{proposition}

\begin{proof}
It is straightforward to verify that $F_\alpha$ is also a sheaf. Lemma~\ref{replem} and the injectivity assumption on $F$ then complete the proof if we can show that every pair of elements $(\beta_1,\beta_2)\in F_\alpha(\Box)\times F_\alpha(\Box)$ actually comes from an element of $F_\alpha(\Box\times\Box)$. To this end, we write $\beta_1 = F(f_1)(\alpha)$ and $\beta_2 = F(f_2)(\alpha)$ for certain $f_1,f_2:W\to\Box$. Now considering $\alpha$ transported along the pairing $(f_1,f_2) : W\to \Box\times\Box$ results in an element of $F(\Box\times\Box)$ that reproduces $(\beta_1,\beta_2)$.
\end{proof}

Here is another smallness result:

\begin{proposition}
$\Sh(\CHs)$ is well-powered.
\label{wellpowered}
\end{proposition}

\begin{proof}
Let $\eta:F\to G$ be a monomorphism in $\Sh(\CHs)$. Then upon composing morphisms of the form $-(C(\Box))\longrightarrow F$ with $\eta$, the Yoneda lemma~\eqref{yonedaeq} shows that the component $\eta_\Box:F(\Box)\to G(\Box)$ is injective, since the diagram
\[
\xymatrix{	\Sh(\CHs)(-(C(\Box)),F) \ar[r]_(.7){\cong}^(.7){\eqref{yonedaeq}} \ar[d]_{\eta\circ -}	&	F(\Box)	\ar[d]^{\eta_\Box}	\\
		\Sh(\CHs)(-(C(\Box)),G) \ar[r]_(.7){\cong}^(.7){\eqref{yonedaeq}} 			&	G(\Box)				}
\]
commutes. 

Again using the sheaf condition on all functions $X\to\Box$ and the fact that $\Box$ is a coseparator in $\CHs$, we can identify the $\alpha\in F(X)$ with the families $\{\beta_f\}$ with $\beta_f\in F(\Box)$ that are indexed by $f:X\to\Box$ and satisfy the compatibility condition that $F(h)(\beta_f) = \beta_{hf}$ for all $f$ and $h:\Box\to\Box$. Hence we have the diagram
\[
\xymatrix{	F(X) \ar[r] \ar[d]^{\eta_X}	& \mathlarger{\mathlarger{\prod}}_{f:X\to\Box} F(\Box) \ar[d]^{\prod_f \eta_{\Box}} \ar@<1ex>[r] \ar@<-1ex>[r]	& \mathlarger{\mathlarger{\prod}}_{f:X\to\Box,h:\Box\to\Box} F(\Box) \ar[d]^{\prod_{f,h} \eta_\Box}	\\
		G(X) \ar[r]			& \mathlarger{\mathlarger{\prod}}_{f:X\to\Box} G(\Box) \ar@<1ex>[r] \ar@<-1ex>[r]					& \mathlarger{\mathlarger{\prod}}_{f:X\to\Box,h:\Box\to\Box} G(\Box)				}
\]
in which both rows are equalizers. So for fixed $G$, the set $F(X)$ is determined by the inclusion map $\eta_\Box:F(\Box)\to G(\Box)$. Hence the number of subobjects of $G$ is bounded by $2^{|G(\Box)|}$.
\end{proof}

\begin{corollary}
Every sheaf $F:\CHs\to\Sets$ for which~\eqref{timesinj} is injective is a (small) colimit in $\Sh(\CHs)$ of representable functors.
\end{corollary}

\begin{proof}
We show that $F$ is the colimit in $\Sh(\CHs)$ of the subfunctors of the form $F_\alpha$ from Proposition~\ref{singlygenerated}, as ordered by inclusion; thanks to Proposition~\ref{wellpowered}, this colimit is equivalent to a small colimit.

To show the required universal property, suppose first that $\eta,\eta'\in\Sh(\CHs)(F,G)$ coincide upon restriction to all $F_\alpha$. Then in particular, $\eta_\Box(\alpha) = \eta'_\Box(\alpha)$ for all $\alpha\in F(\Box)$, and hence $\eta=\eta'$ by the previous results. Conversely, let $\{\phi^\alpha\}_\alpha$ be a family of natural transformations $\phi^\alpha: F_\alpha\to G$ that are compatible in the sense that if $F_\beta\subseteq F_\alpha$, then $\phi^\alpha |_{F_\beta} = \phi^\beta$. Then define the component $\eta_X : F(X)\to G(X)$ on every $\alpha\in F(X)$ as
\[
\eta_X(\alpha)\defin \phi^\alpha_X(\alpha).
\]
The commutativity of the naturality square
\[
\xymatrix{	F(X) \ar[r]^{\eta_X} \ar[d]_{F(f)}	& G(X) \ar[d]^{G(f)}	\\
		F(Y) \ar[r]_{\eta_Y}			& G(Y)			}
\]
on some $\alpha\in F(X)$ follows from  
\[
G(f)(\phi^\alpha_X(\alpha)) = \phi^\alpha_Y(F(f)(\alpha)) = \phi^{F(f)(\alpha)}_{Y}(F(f)(\alpha)) ,
\]
where the first equation is naturality of $\phi^\alpha$ and the second one is the assumed compatibility.

To see that $\eta$ restricts to $\phi^\alpha$ on every $F_\alpha$, we show that the components coincide, i.e.~$\eta_Y = \phi^\alpha_Y$ for all $Y\in\CHs$ and $\beta\in F_\alpha(Y)$. Now we must have $\beta = F(f)(\alpha)$ for suitable $f:X\to Y$, and consider the diagram
\[
\xymatrix{	F_\alpha(X) \ar[r] \ar[d]	& F(X) \ar[r]^{\eta_X} \ar[d]|{F(f)}	& G(X) \ar[d]^{G(f)}	\\
		F_\alpha(Y) \ar[r]		& F(Y) \ar[r]_{\eta_Y}			& G(Y)			}
\]
Starting with $\alpha$ in the upper left, we have $\phi^\alpha_X(\alpha)$ in the upper right, and hence
\[
G(f)(\phi^\alpha_X(\alpha)) = \phi^\alpha_Y(F(f)(\alpha)) = \phi^\alpha_Y(\beta)
\]
in the lower right, where the first equation is as above. Since we also have $\beta$ in the lower left, we obtain the desired $\eta_Y(\beta) = \phi^\alpha_Y(\beta)$.
\end{proof}

In light of the upcoming Theorem~\ref{pCalgthm}, this result is closely related to~\cite[Theorem~5]{piecewise}. The only potential difference is that our colimit is taken in $\Sh(\CHs)$, while van den Berg and Heunen consider it in $\pCalg$, and it is not clear whether these are equivalent.

\bigskip

Since it is currently unclear whether Definition~\ref{sheafdef} is the most adequate collection of sheaf conditions that one can postulate, we do not investigate the categorical properties of $\Sh(\CHs)$ any further in this paper.

\newpage
\section{Piecewise C*-algebras as sheaves $\CHs\to\Sets$}
\label{piecewisesec}

In this section, we will establish that $\Sh(\CHs)$ contains the category of \emph{piecewise C*-algebras} introduced by van den Berg and Heunen~\cite{piecewise} as a full subcategory. The following definition was inspired by Kochen and Specker's consideration of partial algebras~\cite{KS}.\footnote{For this reason van den Berg and Heunen introduced their definition as \emph{partial C*-algebras}, but the term was subsequently changed to \emph{piecewise C*-algebra}~\cite{HR}.}

\newcommand{\commute}{\mathop{\Perp}}

\begin{definition}[\cite{piecewise}]
\label{piecewisedef}
A \emph{piecewise C*-algebra} is a set $A$ equipped with the following pieces of structure:
\begin{enumerate}
\item a reflexive and symmetric relation $\commute\subseteq A\times A$. If $\alpha\commute\beta$, we say that $\alpha$ and $\beta$ \emph{commute};
\item binary operations $+,\cdot:\commute\rightarrow A$;
\item a scalar multiplication $\cdot:\Cl\times A\rightarrow A$;
\item distinguished elements $0, 1\in A$;
\item an involution $*:A\rightarrow A$;
\item a norm $||-||:A\rightarrow \Rl$;
\end{enumerate}
such that every subset $C\subseteq A$ of pairwise commuting elements is contained in some subset $\bar{C}\subseteq A$ of pairwise commuting elements which is a commutative C*-algebra with respect to the data above.
\end{definition}

The piecewise C*-algebras in which the relation $\commute$ is total are precisely the commutative C*-algebras $C(X)$. Our choice of the symbol ``$\commute$'' is explained by the special case of rank one projections, which commute if and only if they are either orthogonal ($\perp$) or parallel ($\parallel$).

\begin{definition}[\cite{piecewise}]
\label{phomdef}
Given piecewise $C^*$-algebras $A$ and $B$, a \emph{piecewise $*$-homomorphism} is a function $\zeta:A\rightarrow B$ such that
\begin{enumerate}
\item If $\alpha\commute\beta$ in $A$, then 
\be
\label{phom}
\zeta(\alpha) \commute \zeta(\beta),\qquad \zeta(\alpha\beta)=\zeta(\alpha)\zeta(\beta),\qquad \zeta(\alpha+\beta)=\zeta(\alpha)+\zeta(\beta).
\ee
\item $\zeta(z\alpha) = z \zeta(\alpha)$ for all $a\in A $ and $z\in \Cl$,
\item $\zeta(\alpha^*) = \zeta(\alpha)^*$ for all $\alpha\in A$.
\item $\zeta(1)=1$.
\end{enumerate}
\end{definition}

\begin{example}
It is well-known that there is no $*$-homomorphism $M_n\to\Cl$ for $n\geq 2$. The Kochen-Specker theorem~\cite{KS} states that for $n\geq 3$, there does not even exist a piecewise $*$-homomorphism $M_n\to\Cl$.
\end{example}

So piecewise $C^*$-algebras and piecewise $*$-homomorphisms form a category $\pCalg$. Still following~\cite{piecewise}, there is a forgetful functor $\Cl(-) : \Calg\to\pCalg$ sending every C*-algebra $A$ to its normal part,
\be
\label{normalpart}
\Cl(A)=\{\: \alpha\in A \:|\: \alpha\alpha^*=\alpha^*\alpha \:\}.
\ee
This set forms a piecewise C*-algebra by postulating that $\alpha \commute \beta$ holds whenever $\alpha$ and $\beta$ commute. $\Cl(-)$ is easily seen to be a faithful functor that reflects isomorphisms. In the language of \emph{property, structure and stuff}~\cite{propertystructurestuff}, this means that it forgets at most structure. So we may think of a C*-algebra as a piecewise C*-algebra together with additional structure, namely the specifications of sums and products of noncommuting elements.

\begin{example}
For $A,B\in\Calg$, any Jordan homomorphism $\Rl(A)\to \Rl(B)$ extends linearly to a piecewise $*$-homomorphism $\Cl(A)\to\Cl(B)$. For example, the transposition map $-^T:M_n\to M_n$ yields a piecewise $*$-homomorphism $\Cl(M_n)\to\Cl(M_n)$.
\end{example}

The discussion of Section~\ref{Calgsasfunctors} extends canonically to piecewise C*-algebras. To wit, Gelfand duality still implements an equivalence of $\CHs^\op$ with a full subcategory of $\pCalg$, so that for every $A\in\pCalg$ we can restrict the hom-functor
\[
\pCalg(-,A) \: : \: \pCalg^\op\to\Sets
\]
to a functor $\CHs\to\Sets$, which maps $X\in\CHs$ to the set of piecewise $*$-homomorphisms $C(X)\to A$. For $A\in\Calg$, this results precisely in the functor $\CHs\to\Sets$ that we already know from Section~\ref{Calgsasfunctors}, since then $\pCalg(C(X),A)=\Calg(C(X),A)$. In other words, we have a diagram of functors
\[
\xymatrix{	\Calg \ar[dr]_{\Cl(-)} \ar[rr]	&		& \Sets^{\CHs}	\\
						& \pCalg\ar[ur]	&		}
\]
In fact, the proof of Proposition~\ref{guarcommcrit} still goes through for piecewise C*-algebras. Hence the functor $\pCalg\to\Sets^\CHs$ actually lands in the full subcategory $\Sh(\CHs)$ as well, and the commutative triangle of functors can be taken to be
\[
\xymatrix{	\Calg \ar[dr]_{\Cl(-)} \ar[rr]	&		& \Sh(\CHs)	\\
						& \pCalg\ar[ur]	&		}
\]
We now investigate the functor on the right a bit further, finding that it is close to being an equivalence. In the following, we use the unit disk $\bigcirc\subseteq\Cl$. Since it is homeomorphic to the unit square $\Box$ that we have been working with until now, all previous statements apply likewise with $\Box$ replaced by $\bigcirc$.

\begin{theorem}
The functor $\pCalg\longrightarrow\Sh(\CHs)$ is fully faithful, with essential image given by all those $F\in\Sh(\CHs)$ for which the canonical map
\be
\label{circtimes}
\xymatrix{	F(\bigcirc\times\bigcirc) \ar[r]	& F(\bigcirc)\times F(\bigcirc)	}
\ee
is injective.
\label{pCalgthm}
\end{theorem}

So this functor forgets at most property, namely the property of injectivity of~\eqref{circtimes} as investigated in Lemma~\ref{injlem}. This property is equivalent to $F$ being separated (in the presheaf sense) on the effective-monic cones. It seems natural to suspect that not every sheaf on $\CHs$ is separated in this sense, but this remains open. So it is also conceivable that $\pCalg\to\Sh(\CHs)$ actually is an equivalence of categories.

In particular, this shows that $\cCalg$ is dense in $\pCalg$, i.e.~that the canonical functor $\pCalg\to\Sets^{\cCalg^{\op}}$ is fully faithful. For a potentially related result of a similar flavour, see~\cite[Corollary~8]{SU}.

\begin{proof}
A piecewise $*$-homomorphism $\zeta:A\to B$ is determined by its action on the unit ball, which is the set of elements with spectrum in $\bigcirc$. In particular, $\zeta$ is uniquely determined by the associated transformation $-(\zeta) : -(A)\to -(B)$, so that the functor under consideration is faithful.

Concerning fullness, let $\eta : -(A)\to -(B)$ be a natural transformation. Its component at $\bigcirc$ is a map $\eta_\bigcirc : \bigcirc(A)\to \bigcirc(B)$. The pairs of commuting elements $\alpha,\beta\in \bigcirc(A)$ are precisely those that are in the image of the canonical map
\[
(\bigcirc\times\bigcirc)(A)\longrightarrow \bigcirc(A)\times \bigcirc(A),
\]
and hence the requirements~\eqref{phom} follow from naturality and the consideration of functions like~\eqref{addition} and~\eqref{multiplication}. The other axioms are likewise simple consequences of naturality. This exhibits a piecewise $*$-homomorphism $\zeta:A\to B$ such that $\eta_\bigcirc$ coincides with $\bigcirc(\zeta)$. Then by Lemma~\ref{evalbox}, we have $\eta = -(\zeta)$.

Finally, we show that every $F\in\Sh(\CHs)$ for which~\eqref{circtimes} is injective is isomorphic to $-(A)$ for some $A\in\pCalg$. Concretely, we construct a piecewise C*-algebra $A$ by first defining its unit ball to be
\[
\bigcirc(A) \defin F(\bigcirc).
\]
This set comes equipped with a commutation relation: $\alpha\commute\beta$ is declared to hold for $\alpha,\beta\in A$ precisely when $(\alpha,\beta)$ is in the image of~\eqref{circtimes}. In this case, we can define the sum $\alpha+\beta$ and the product $\alpha\beta$ using the functoriality on maps such as~\eqref{addition} and~\eqref{multiplication}. Likewise there is a scalar multiplication by numbers $z\in \bigcirc$ and an involution arising from functoriality on the complex conjugation map $\bigcirc\to\bigcirc$.

Now defining $A$ to consist of pairs $(\alpha,z)\in F(\bigcirc)\times \Rl_{>0}$, modulo the equivalence $(\alpha,z)\sim(s a,s z)$ for all $s\in(0,1)$, results in a piecewise C*-algebra: the relevant structure of Definition~\ref{piecewisedef} extends canonically from $\bigcirc(A)$ to all of $A$, and we also claim that any set $\{\gamma_i\}_{i\in I}\subseteq \bigcirc(A)$ of pairwise commuting elements is contained in a commutative C*-subalgebra. We write this family as a single element of the $I$-fold product,
\[
\gamma \in F(\bigcirc)^I.
\]
The cone $\{(p_i,p_j) : \bigcirc^I\to\bigcirc\times\bigcirc\}_{i,j\in I}$ consisting of all pairings of projections $p_i : \bigcirc^I\to \bigcirc$ is effective-monic and directed. By the commutativity assumption on $\gamma$, the pair $(\gamma_i,\gamma_j)\in F(\bigcirc)\times F(\bigcirc)$ comes from an element of $F(\bigcirc\times\bigcirc)$. Hence by the sheaf condition, $\gamma$ is actually the image of an element $\gamma'\in F\left(\bigcirc^I\right)$ under the canonical map. The subfunctor $F_{\gamma'}\subseteq F$, as in Proposition~\ref{singlygenerated}, is representable. It corresponds to the commutative C*-subalgebra generated by the $\gamma_i$.
\end{proof}

The following criterion---due to Heunen and Reyes---describes the image of the functor $\Cl(-):\Calg\to\pCalg$ at the level of morphisms.

\begin{lemma}[{\cite[Proposition~4.13]{HR}}]
\label{fextend}
For $A,B\in\Calg$, a piecewise $*$-homomorphism $\zeta:\Cl(A)\to \Cl(B)$ extends to a $*$-homomorphism $A\to B$ if and only if it is additive on self-adjoints and multiplicative on unitaries.
\end{lemma}

By faithfulness of $\Calg\to\pCalg$, we already know such an extension to be unique if it exists.

\begin{proof}
The `only if' part is clear, so we focus on the `if' direction. Every element of $A$ is of the form $a+ib$ for $a,b\in \Rl(A)$, and linearity forces us to define the candidate extension of $f$ by
\[
\hat{\zeta}(a+ib) \defin \zeta(a) + i\zeta(b).
\]
In this way, $\hat{\zeta}$ becomes linear due to the first assumption, and is evidently involutive and unital. On a unitary $\nu$, we have $\hat{\zeta}(\nu)=\zeta(\nu)$, since
\begin{align*}
\hat{\zeta}(\nu) 	& = \hat{\zeta}\left( \frac{\nu+\nu^*}{2} + i\frac{\nu-\nu^*}{2}\right) = \frac{1}{2}\zeta(\nu+\nu^*) + \frac{1}{2}\zeta(\nu-\nu^*) \\
 		& = \frac{1}{2}\zeta(\nu) + \frac{1}{2}\zeta(\nu^*) + \frac{1}{2}\zeta(\nu) - \frac{1}{2}\zeta(\nu^*) = \zeta(\nu),
\end{align*}
where the third step uses $\nu\commute \nu^*$.

We finish the proof by arguing that $\hat{\zeta}$ is multiplicative on two arbitrary elements $\alpha,\beta\in [-1,+1](A)$, which is enough to prove multiplicativity generally, and hence to show that $\hat{\zeta}$ is indeed a $*$-homomorphism. By functional calculus, we can find unitaries $\nu,\tau\in \Tl(A)$ such that $\alpha=\nu+\nu^*$ and $\beta=\tau+\tau^*$. Then
\begin{align*}
\hat{\zeta}(\alpha\beta) = \hat{\zeta}\left( (\nu+\nu^*)(\tau+\tau^*) \right) 	& = \hat{\zeta}\left( \nu\tau + \nu\tau^* + \nu^* \tau + \nu^* \tau^* \right) \\
					& = \hat{\zeta}(\nu\tau) + \hat{\zeta}(\nu\tau^*) + \hat{\zeta}(\nu^*\tau) + \hat{\zeta}(\nu^* \tau^*) \\
					& = \zeta(\nu\tau) + \zeta(\nu\tau^*) + \zeta(\nu^*\tau) + \zeta(\nu^*\tau^*) \\
					& = \zeta(\nu)\zeta(\tau) + \zeta(\nu)\zeta(\tau^*) + \zeta(\nu^*)\zeta(\tau) + \zeta(\nu^*)\zeta(\tau^*) \\
					& = (\zeta(\nu) + \zeta(\nu^*))(\zeta(\tau) + \zeta(\tau^*)) \\
					& = (\hat{\zeta}(\nu) + \hat{\zeta}(\nu^*))(\hat{\zeta}(\tau) + \hat{\zeta}(\tau^*)) \\
					& = \hat{\zeta}(\nu + \nu^*) \hat{\zeta}(\tau + \tau^*) = \hat{\zeta}(\alpha)\hat{\zeta}(\beta),
\end{align*}
where we have used that $\hat{\zeta}$ coincides with $\zeta$ on unitaries (third and sixth line) and the assumption of multiplicativity on unitaries (fourth line).
\end{proof}

In fact, this result can be improved upon:

\begin{proposition}
A piecewise $*$-homomorphism $\zeta:\Cl(A)\to\Cl(B)$ extends to a $*$-homomorphism $A\to B$ if and only if it is multiplicative on unitaries.
\label{fextend2}
\end{proposition}

\begin{proof}
By the lemma, it is enough to prove that such a $\zeta$ is additive on self-adjoints.

We use the following fact, which follows from the exponential series: for every $\alpha,\beta\in\Rl(A)$ and real parameter $t\in\Rl$, the unitary
\[
e^{it(\alpha + \beta)} e^{-it\alpha} e^{-it\beta}
\]
differs from $1$ by at most $O(t^2)$ as $t\to 0$. Since $\zeta$ preserves the spectrum of unitaries, we conclude that also
\[
\zeta\left(e^{it(\alpha + \beta)} e^{-it\alpha} e^{-it\beta}\right) = e^{it\zeta(\alpha + \beta)} e^{-it\zeta(\alpha)} e^{-it\zeta(\beta)}
\]
is a unitary that differs from $1$ by at most $O(t^2)$. By the same argument as above, this implies $\zeta(\alpha + \beta) = \zeta(\alpha) + \zeta(\beta)$, as was to be shown.
\end{proof}

As the proof shows, we actually only need multiplicativity on products of exponentials, i.e.~on the connected component of the identity $1\in\Tl(A)$. Also, the method of proof suggests a relation to the Baker-Campbell-Hausdorff formula, which may be worth exploring further.

Finally, it is worth noting that a piecewise $*$-homomorphism $\zeta:\Cl(A)\to\Cl(B)$ is additive on self-adjoints if and only if it is a Jordan homomorphism: the condition $\zeta(\alpha^2) = \zeta(\alpha)^2$ for $\alpha\in\Rl(A)$ is automatic since $\zeta$ preserves functional calculus.

Let us end this section by stating its main open problem:

\begin{problem}
\label{pCalgprob}
Is the functor $\pCalg\to\Sh(\CHs)$ an equivalence of categories, i.e.~does \emph{every} sheaf on $\CHs$ satisfy the injectivity condition of Lemma~\ref{injlem}?
\end{problem}

\newpage
\section{Almost C*-algebras as piecewise C*-algebras with self-action}
\label{secsaC}

What we have learnt so far is that considering a C*-algebra $A$ as a sheaf $-(A):\CHs\to\Sets$, or equivalently as a piecewise C*-algebra, recovers the entire `commutative part' of the C*-algebra structure of $A$. Nevertheless, the functor $\Calg\to\pCalg$ is not full, which indicates that part of the relevant structure is lost: for example, a C*-algebra $A$ is in general not isomorphic to $A^\op$~\cite{phillips}, although the two are canonically isomorphic as piecewise C*-algebras. This raises the question: which natural piece of additional structure on a sheaf $\CHs\to\Sets$ or piecewise C*-algebra would let us recover the missing information?

Of course, what kind of additional structure counts as `natural' is a subjective matter. But again, we can take inspiration from quantum physics: which additional structure would have a clear physical interpretation? Our following proposal is based on a central feature of quantum mechanics: observables generate dynamics, in the sense that to every observable (self-adjoint operator) $\alpha\in\Rl(A)$, one associates the one-parameter group of inner automorphisms given by
\be
\label{autos}
\Rl\times A \longrightarrow A,\qquad (t,\beta) \longmapsto e^{i\alpha t} \beta e^{-i\alpha t}.
\ee
For example, if $\alpha$ is energy, then the resulting one-parameter family of automorphisms is given precisely by time translations, i.e.~by the inherent dynamics of the system under consideration. If $\alpha$ is a component of angular momentum, then the resulting family of automorphisms are the rotations around that axis. As is obvious from~\eqref{autos}, this natural way in which $A$ acts on itself by inner automorphisms is a purely noncommutative feature, in that it becomes trivial in the commutative case.

More formally, the construction of~\eqref{autos} really consists of two parts: first, for every $t\in\Rl$, one forms the unitary $\nu\defin e^{-i\alpha t}$; since this is functional calculus, it is captured by the functoriality $\CHs\to\Sets$. Second, one lets $\nu$ act on $A$ via conjugation, as $\beta\mapsto \nu^* \beta\nu$. This part is not captured by what we have discussed so far, and hence we axiomatize it as an additional piece of structure. Our definition is similar in spirit to the `active lattices' of Heunen and Reyes~\cite{HR} and also seems related to~\cite[Section~VI]{BMU}.

\begin{definition}
\label{almostdef}
An \emph{almost C*-algebra} is a pair $(A,\mathfrak{a})$ consisting of a piecewise C*-algebra $A\in\pCalg$ and a \emph{self-action} of $A$, which is a map
\[
\mathfrak{a} : \Tl(A) \longrightarrow \pCalg(A,A)
\]
assigning to every unitary $\nu\in \Tl(A)$ a piecewise automorphism $\mathfrak{a}(\nu) : A \to A$ such that  
\begin{itemize}
	\item $\nu$ commutes with $\tau\in\Tl(A)$ if and only if $\mathfrak{a}(\nu)(\tau) = \tau$;
	\item in this case, $\mathfrak{a}(\nu\tau) = \mathfrak{a}(\nu)\mathfrak{a}(\tau)$.
\end{itemize}
\end{definition}

So $\mathfrak{a}$ must satisfy two equations on commuting unitaries. The first equation implies that a commutative C*-algebra, considered as a piecewise C*-algebra, can act on itself only trivially; and conversely, if the self-action is trivial in the sense that every $\mathfrak{a}(\nu)$ is the identity, then $A$ must be commutative. The second equation implies that if $\nu$ and $\tau$ commute, then also their actions commute:
\[
\mathfrak{a}(\nu)\mathfrak{a}(\tau) = \mathfrak{a}(\nu\tau) = \mathfrak{a}(\tau\nu) = \mathfrak{a}(\tau)\mathfrak{a}(\nu).
\]
While introducing a self-action $\mathfrak{a} : \Tl(A)\longrightarrow\pCalg(A,A)$ can be physically motivated by the above discussion; we expect the appearance of $\Tl$ to be related to Pontryagin duality. The physical interpretation of the first axiom could be related to Noether's theorem.

Almost C*-algebras form a category denoted $\aCalg$ as follows:

\begin{definition}
\label{ahomdef}
An \emph{almost $*$-homomorphism} $\zeta:(A,\mathfrak{a})\to (B,\mathfrak{b})$ is a piecewise $*$-homomorphism $\zeta:A\to B$ which preserves the self-actions in the sense that
\be
\label{ahom}
\mathfrak{b}(\zeta(\nu)) (\zeta(\alpha)) = \zeta(\mathfrak{a}(\nu)(\alpha)).
\ee
\end{definition}

The forgetful functor $\Calg\to\pCalg$ factors through $\aCalg$ by associating to every C*-algebra $A$ and unitary $\nu\in \Tl(A)$ its conjugation action,
\[
\mathfrak{a}(\nu)(\alpha) \defin \nu^* \alpha \nu.
\]
Every $*$-homomorphism $\zeta:A\to B$ is compatible with the resulting self-actions: the condition~\eqref{ahom} becomes simply
\be
\label{presconj}
\zeta(\nu)^* \zeta(\alpha) \zeta(\nu) = \zeta(\nu^* \alpha \nu).
\ee

Our main question is whether the additional structure of a self-action that is present in an almost C*-algebras is sufficient to recover the entire C*-algebra structure:

\begin{problem}
	\label{aCalgprob}
Is the forgetful functor $\Calg\to\aCalg$ an equivalence of categories?
\end{problem}

In order for this to be the case, one would have to show that the functor is both fully faithful and essentially surjective. While the latter question is wide open, it is clear that the functor is faithful, since already the forgetful functor $\Calg\to\pCalg$ is. We can also prove fullness in a W*-algebra setting:

\begin{theorem}
\label{Wfull}
$\Calg\to\aCalg$ is fully faithful on morphisms out of any W*-algebra.
\end{theorem}

This result is similar to~\cite[Theorem~4.11]{HR}, but does not directly follow from it\footnote{This is because the notion of `active lattice' of~\cite{HR} includes a group that acts on the lattice, and a morphism of active lattices in particular is \emph{assumed} to be a homomorphism of the corresponding groups. If we assumed something analogous in our definition of almost C*-algebra, the fullness of the forgetful functor would simply follow from Proposition~\ref{fextend2}.}.

\begin{proof}
We need to show surjectivity, i.e.~if $\zeta:\Cl(A)\to \Cl(B)$ for a W*-algebra $A$ is a piecewise $*$-homomorphism which satisfies~\eqref{presconj}, then $\zeta$ extends to a $*$-homomorphism $A\to B$. Let us first consider the case that $A$ contains no direct summand of type $I_2$. Then for every state $\phi:B\to\Cl$, the map
\be
\label{quasifun}
\alpha + i\beta \longmapsto \phi(\zeta(\alpha) + i\zeta(\beta))
\ee
for $\alpha,\beta\in \Rl(A)$ is a quasi-linear functional on $A$ in the sense of~\cite[Definition~5.2.5]{hamhalter}, and therefore is uniquely determined by its values on the projections $\mathbf{2}(A)$~\cite[Proposition~5.2.6]{hamhalter}. On the other hand, by the generalized Gleason theorem~\cite[Theorem~5.2.4]{hamhalter}, this map $\mathbf{2}(A)\to\Rl$ uniquely extends to a state $A\to\Rl$. In conclusion, composition with $\zeta$ takes states on $B$ to states on $A$, and hence $\Rl(\zeta):\Rl(A)\to\Rl(B)$ is linear.

On $\Rl(A)$, we furthermore have $\zeta(\alpha^2) = \zeta(\alpha)^2$, which makes $\zeta$ into a Jordan homomorphism. By a deep result of St{\o}rmer~\cite[Theorem~3.3]{stormer}, this means that there exists a projection $\pi\in \mathbf{2}(B)$, commuting with the range of $\zeta$, such that $\alpha\mapsto \pi\zeta(\alpha)$ uniquely extends to a (generally nonunital) $*$-homomorphism, and similarly $\alpha\mapsto (1-\pi)\zeta(\alpha)$ uniquely extends to a (generally nonunital) $*$-anti-homomorphism. In other words, $\zeta$ decomposes into the sum of the restriction (to normal elements) of a $*$-homomorphism and a $*$-anti-homomorphism. So far, we have only made use of the assumption that $\zeta$ is a piecewise $*$-homomorphism. 

Hence in order to complete the proof in the case of $A$ without type $I_2$ summand, working with the corner $(1-\pi)A(1-\pi)$ in place of $A$ itself shows that it is enough to consider the case $\pi=0$, i.e.~that $\zeta$ is the restriction of a $*$-anti-homomorphism. In particular,
\[
\zeta(\nu)^* \zeta(\alpha) \zeta(\nu) \stackrel{\eqref{presconj}}{=} \zeta(\nu^*\alpha\nu) = \zeta(\nu) \zeta(\alpha) \zeta(\nu)^*, 
\]
and therefore $\zeta(\alpha) \zeta(\nu^2) = \zeta(\nu^2) \zeta(\alpha)$ for all $\nu\in\Tl(A)$ and $\alpha\in\Cl(A)$. Since every exponential unitary $e^{i\beta}$ is the square of another unitary, we know that $\zeta(\alpha)$ commutes with every exponential unitary. Since every element of $A$ is a linear combination of exponential unitaries, we conclude that $\zeta(\alpha)$ commutes with $\zeta(\beta)$ for every $\beta\in\Cl(A)$. Hence the range of $\zeta$ is commutative. In particular, $\zeta$ is also the restriction of a $*$-homomorphism, which completes the proof in the present case.

Now consider the case of an almost $*$-homomorphism $\zeta:\Cl(M_2)\to \Cl(B)$. Due to the isomorphism $M_2\cong \mathrm{Cl}(\Rl^2)\otimes\Cl$ with a complexified Clifford algebra, $M_2$ is freely generated as a C*-algebra by two self-adjoints $\sigma_x$ and $\sigma_y$ subject to the relations
\[
\sigma_x^2 = \sigma_y^2 = 1,\qquad \sigma_x \sigma_y + \sigma_y \sigma_x = 0.
\]
Since $\zeta$ commutes with functional calculus, the first two equations are clearly preserved by $\zeta$ in the sense that $\zeta(\sigma_x)^2 = \zeta(\sigma_y)^2 = 1$. Concerning the third equation, we know
\[
-\zeta(\sigma_x) = \zeta(-\sigma_x) = \zeta(\sigma_y \sigma_x \sigma_y) \stackrel{\eqref{presconj}}{=} \zeta(\sigma_y) \zeta(\sigma_x) \zeta(\sigma_y) .
\]
Hence $\zeta(\sigma_x) \zeta(\sigma_y) + \zeta(\sigma_y) \zeta(\sigma_x) = 0$ due to $\zeta(\sigma_y)^2 = 1$. Therefore the values $\zeta(\sigma_x)$ and $\zeta(\sigma_y)$ extend uniquely to a $*$-homomorphism $\hat{\zeta} : M_2\to B$; the problem is to show that this coincides with the original $\zeta$ on normal elements. Since any symmetry $\nu\in \{-1,+1\}(M_2)$ is conjugate to $\sigma_x$, we certainly have $\hat{\zeta}(\nu) = \zeta(\nu)$ by~\eqref{presconj} and the assumption $\hat{\zeta}(\sigma_x) = \zeta(\sigma_x)$. But because in the special case of $M_2$, every normal element can be obtained from a symmetry by functional calculus, and both $\zeta$ and $\hat{\zeta}$ preserve functional calculus, this is sufficient to show that $\hat{\zeta} = \zeta$ on normal elements. This finishes off the case $A=M_2$.

A general W*-algebra of type $I_2$ is of the form $A\cong L^\infty(\Omega,\mu,M_2)$ for a suitable measure space $(\Omega,\mu)$. Let $\zeta:\Cl(A)\to\Cl(B)$ be an almost $*$-homomorphism. We first show that $\zeta$ uniquely extends to a bounded $*$-homomorphism on the *-subalgebra of simple functions. For a measurable set $\Gamma\subseteq\Omega$, let $\chi_\Gamma : \Omega\to\{0,1\}$ be the associated indicator function. For nonempty $\Gamma$, the algebra elements of the form $\alpha\chi_\Gamma$ for $\alpha\in M_2$ form a C*-subalgebra isomorphic to $M_2$ itself (with different unit). By the previous, we know that $\zeta$ uniquely extends to a $*$-homomorphism on this subalgebra. Furthermore, $\zeta$ behaves as expected on a simple function $\sum_{i=1}^n \alpha_i \chi_{\Gamma_i}$: assuming that the $\Gamma_i$'s form a partition of $\Omega$, we have $\alpha_i \chi_{\Gamma_i} \cdot \alpha_j \chi_{\Gamma_j} = 0$ for $i\neq j$, and hence $\zeta$ is additive on the sum, which implies
\be
\label{simpleadd}
\zeta\left( \sum_{i=1}^n \alpha_i \chi_{\Gamma_i} \right) = \sum_{i=1}^n \zeta(\alpha_i) \zeta(\chi_{\Gamma_i}).
\ee
We show that $\zeta$ is linear on the sum of two self-adjoint simple functions. By choosing a common refinement, it is enough to consider the case that the two partitions are the same. But then additivity follows from~\eqref{simpleadd} and additivity on $M_2$. Multiplicativity on unitary simple functions is analogous. Since the proof of Lemma~\ref{fextend} still goes through in the present situation (where the *-algebra of simple functions is generally not a C*-algebra), we conclude that $\zeta$ extends uniquely to a $*$-homomorphism on the simple functions. By construction, this $*$-homomorphism is bounded. Therefore it uniquely extends to a $*$-homomorphism $\hat{\zeta}:A\to B$ which coincides with $\zeta$ on the normal simple functions. It remains to be shown that $\hat{\zeta}(\alpha) = \zeta(\alpha)$ for all $\alpha\in \Cl(A)$.

To obtain this for a given $\alpha\in\Cl(A)$, we distinguish those points $x\in\Omega$ for which $\alpha(x)$ is degenerate from those for which it is not. Since degeneracy is detected by the vanishing of the discriminant $\tr^2 - 4\det$, the relevant set is
\[
\Delta \defin \{\: x\in\Omega \:|\: \tr(\alpha(x))^2 - 4 \det(\alpha(x)) = 0 \:\}.
\]
This set is measurable since both trace and determinant are measurable functions $M_2\to\Cl$. For every $x\in\Omega\setminus\Delta$, there is a unique unitary $\nu(x)\in\Tl(M_2)$ such that $\nu(x)^* \alpha(x) \nu(x)$ is diagonal. Since the eigenbasis of a nondegenerate self-adjoint matrix depends continuously on the matrix, it follows that the function $x\mapsto \nu(x)$ is also measurable. By arbitrarily choosing $\nu(x)\defin 1$ on $x\in\Delta$, we have constructed a unitary $\nu\in\Tl(L^\infty(\Omega,\mu,M_2))$ such that $\nu^* \alpha \nu$ is pointwise diagonal. Thanks to~\eqref{presconj}, it is therefore sufficient to prove the desired identity $\hat{\zeta}(\alpha) = \zeta(\alpha)$ on diagonal $\alpha$ only. But since these diagonal elements generate a commutative C*-subalgebra, which contains a dense *-subalgebra of simple functions on which $\hat{\zeta}$ and $\zeta$ are known to coincide, we are done because both $\hat{\zeta}$ and $\zeta$ are $*$-homomorphisms on this commutative subalgebra.

Now a general W*-algebra $A$ is a direct sum of a W*-algebra without $I_2$ summand and one that is of type $I_2$~\cite[Theorems~1.19~\&~1.31]{takesaki}. Again by considering corners, it is straightforward to check that if the fullness property holds on almost $*$-homomorphisms out of $A,B\in\Calg$, then it also holds on almost $*$-homomorphisms out of $A\oplus B$.
\end{proof}

In general, the problem of fullness is related to the cohomology of the unitary group $\Tl(A)$ as follows. Let $\zeta:\Cl(A)\to\Cl(B)$ be an almost $*$-homomorphism between C*-algebras. We can assume without loss of generality that $\im(\zeta)$ generates $B$ as a C*-algebra. For unitaries $\nu,\tau\in\Tl(A)$ and any $\alpha\in\bigcirc(A)$, we have
\begin{align*}
\zeta(\alpha)	& = \zeta\left(\tau^*\nu^*(\nu\tau)\alpha(\nu\tau)^*\nu\tau\right)			\\
		& \stackrel{\eqref{presconj}}{=} \zeta(\tau)^* \zeta(\nu)^* \zeta(\nu\tau) \zeta(\alpha) \zeta(\nu\tau)^* \zeta(\nu) \zeta(\tau)		\\
		& = \: \left( \zeta(\nu\tau)^* \zeta(\nu) \zeta(\tau)\right)^*\: \zeta(\alpha) \:\left( \zeta(\nu\tau)^* \zeta(\nu) \zeta(\tau)\right)
\end{align*}
Hence the unitary $\zeta(\nu\tau)^*\zeta(\nu)\zeta(\tau)$ commutes with $\zeta(\alpha)$. By the assumption that $\im(\zeta)$ generates $B$, this means that there exists $c(\nu,\tau)$ in the centre of $\Tl(B)$ such that
\[
\zeta(\nu\tau) = c(\nu,\tau) \zeta(\nu)\zeta(\tau).
\]
As in the theory of projective representations of groups, we can use this relation to evaluate $\zeta$ on a product of three unitaries $\nu,\tau,\chi\in\Tl(A)$, resulting in
\[
c(\nu\tau,\chi) c(\nu,\tau) \zeta(\nu)\zeta(\tau)\zeta(\chi) = \zeta(\nu\tau\chi) = c(\nu,\tau\chi)c(\tau,\chi) \zeta(\nu)\zeta(\tau)\zeta(\chi).
\]
This establishes the cocycle equation
\[
c(\tau,\chi) c(\nu\tau,\chi)^* c(\nu,\tau\chi) c(\nu,\tau)^* = 1,
\]
showing that $c$ is a $2$-cocycle on $\Tl(A)$ with values in the centre of $\Tl(B)$, which is equal to the unitary group of the centre of $B$. Unfortunately, we do not know whether this can be used to show that $\Tl(\zeta):\Tl(A)\to\Tl(B)$ is a group homomorphism, which would be enough to prove fullness in general by Proposition~\ref{fextend2}.

Let us now restate the remaining part of Problem~\ref{aCalgprob}:

\begin{problem}
Is the functor $\Calg\to\aCalg$ full in general? If so, could it even be essentially surjective?
\end{problem}

\newpage
\section{Groups as piecewise groups with self-action}
\label{secgrps}

In order to get a better intuition for the relation between C*-algebras and almost C*-algebras, it is instructive to perform analogous considerations for other mathematical structures. In this section, we investigate the case of groups, which may also be of interest in its own right.

By analogy with piecewise C*-algebras, we have:

\begin{definition}[\cite{HR}]
\label{piecewisegroupdef}
A \emph{piecewise group} is a set $G$ equipped with the following pieces of structure:
\begin{enumerate}
\item a reflexive and symmetric relation $\commute\subseteq G\times G$. If $x\commute y$, we say that $x$ and $y$ \emph{commute};
\item a binary operation $\cdot:\commute\rightarrow G$;
\item a distinguished element $1\in G$;
\end{enumerate}
such that every subset $C\subseteq G$ of pairwise commuting elements is contained in some subset $\bar{C}\subseteq G$ of pairwise commuting elements which is an abelian group with respect to the data above.
\end{definition}

Abelian groups are precisely those piecewise groups for which the commutativity relation $\commute$ is total. Piecewise groups form a category $\pGrp$ in the obvious way: 

\begin{definition}
\label{pghomdef}
Given piecewise groups $G$ and $H$, a \emph{piecewise group homomorphism} is a function $\zeta:G\rightarrow H$ such that if $g\commute h$ in $G$, then 
\be
\label{pghom}
\zeta(g) \commute \zeta(h),\qquad \zeta(gh)=\zeta(g)\zeta(h).
\ee
\end{definition}

It is straightforward to show that a piecewise group homomorphism satisfies $\zeta(1)=1$.

Considering every group as a piecewise group results in a forgetful functor $\Grp\to\pGrp$, which is faithful and reflects isomorphisms. Since it is not full (taking inverses $g\mapsto g^{-1}$ is a piecewise group homomorphism for every $G$, but a group homomorphism only if $G$ is abelian), this functor forgets some of the structure that groups have. By analogy with Definition~\ref{almostdef}, we try to recover this structure by equipping a piecewise group with a notion of inner automorphisms:

\begin{definition}
\label{almostgroupdef}
An \emph{almost group} is a pair $(G,\mathfrak{a})$ consisting of $G\in\pGrp$ and a \emph{self-action} on $G$, which is a map
\[
\mathfrak{a} : G \longrightarrow \pGrp(G,G)
\]
assigning to every element $g\in G$ a piecewise automorphism $\mathfrak{a}(g) : G \to G$ such that
	\begin{itemize}
		\item $g$ commutes with $h$ if and only if $\mathfrak{a}(g)(h) = h$;
		\item in this case, $\mathfrak{a}(gh) = \mathfrak{a}(g)\mathfrak{a}(h)$.
	\end{itemize}
\end{definition}

Almost groups form a category denoted $\aGrp$ as follows:

\begin{definition}
\label{aghomdef}
An \emph{almost group homomorphism} $\zeta:(G,\mathfrak{a})\to (H,\beta)$ is a piecewise group homomorphism $\zeta:A\to B$ such that
\be
\label{aghom}
\mathfrak{a}(\zeta(g)) (\zeta(h)) = \zeta(\mathfrak{a}(g)(h)).
\ee
\end{definition}

The forgetful functor $\Grp\to\pGrp$ factors through $\aGrp$ by associating to every group $G$ and element $g\in G$ the conjugation action,
\[
\mathfrak{a}(g)(h) \defin g^{-1} h g.
\]
Every group homomorphism $\zeta:G\to H$ respects the resulting self-actions: the condition~\eqref{ahom} becomes simply
\be
\label{presconjgrp}
\zeta(g)^{-1} \zeta(h) \zeta(g) = \zeta(g^{-1} h g).
\ee
One can ask whether this forgetful functor $\Grp\to\aGrp$ is an equivalence of categories. In contrast to the discussion of Section~\ref{secsaC}, and in particular Theorem~\ref{Wfull}, here we know the answer to be negative:

\begin{theorem}
The forgetful functor $\Grp\to\aGrp$ is not full.
\label{aGnotfull}
\end{theorem}

So in general, going from a group to an almost group still constitutes a loss of structure.

\begin{proof}
We provide an explicit example of an almost group homomorphism between groups that is not a group homomorphism.

Let $\Fl_2$ be the free group on two generators $a$ and $b$. For any word $w\in\Fl_2$, let $\hat{w}$ be the cyclically reduced word associated to $w$. Then consider the map $\zeta:\Fl_2\to\Zl$ defined as $\zeta(w)$ being the number of times that the generator $a$ directly precedes the generator $b$ in $\hat{w}$, minus the number of times that the generator $b^{-1}$ directly precedes the generator $a^{-1}$ in $\hat{w}$. By construction, this is invariant under conjugation and therefore satisfies~\eqref{presconjgrp}. If $v,w\in\Fl_2$ commute, then they must be of the form $v=u^m$ and $w=u^n$ for some $u\in\Fl_2$ and $m,n\in\Zl$~\cite[Proposition~2.17]{ls}. Hence to verify that $\zeta$ is a piecewise group homomorphism, it is enough to show that $\zeta(u^k)=\zeta(u)^k$ for all $k\in\Zl$. This is the case because we have $\hat{u^k} = \hat{u}^k$ at the level of reduced cyclic words.

On the other hand, $\zeta$ is not a group homomorphism since $\zeta(a) = \zeta(b) = 0$, while $\zeta(ab) = 1$.
\end{proof}

As the second half of the proof indicates, part of the problem is that a free group has very few commuting elements. One can hope that the situation will be better for finite groups:

\begin{problem}
Is the restriction of the functor $\Grp\to\aGrp$ from finite groups to finite almost groups an equivalence of categories?
\end{problem}

\newpage
\newgeometry{top=2cm}
\bibliographystyle{unsrt}
\bibliography{Calgs_as_sheaves}

\end{document}